\documentclass[reqno]{amsart}
\usepackage{amsmath,amssymb,mathrsfs}
\usepackage{graphicx,float,cite,times}
\usepackage{graphicx,subfigure,caption}
\usepackage{bm}
\usepackage[all]{xy}
\usepackage{float}

\usepackage{comment}
\includecomment{note}
\usepackage[marginpar=1in]{geometry}
\usepackage[
colorlinks,
citecolor=magenta,
urlcolor=blue,
linkcolor=magenta,
allcolors=black,
backref=page
]{hyperref}

\theoremstyle{definition}
\newtheorem{theo}{Theorem}[section]

\newtheorem{lemm}[theo]{Lemma}

\newtheorem{rema}[theo]{Remark}

\numberwithin{equation}{section}

\newcommand{\bC}{{\mathbb{C}}}

\newcommand{\bN}{{\mathbb{N}}}
\newcommand{\bR}{{\mathbb{R}}}
\newcommand{\bZ}{{\mathbb{Z}}}

\newcommand{\tC}{{\mathtt{C}}}
\newcommand{\tD}{{\mathtt{D}}}
\newcommand{\tF}{{\mathtt{F}}}
\newcommand{\tG}{{\mathtt{G}}}
\newcommand{\tI}{{\mathtt{I}}}
\newcommand{\tT}{{\mathtt{T}}}

\newcommand{\calA}{{\mathcal{A}}}

\newcommand{\ca}{{\mathfrak{a}}}
\newcommand{\cA}{{\mathfrak{A}}}
\newcommand{\cB}{{\mathfrak{B}}}
\newcommand{\cc}{{\mathfrak{c}}}
\newcommand{\cd}{{\mathfrak{d}}}
\newcommand{\ce}{{\mathfrak{e}}}

\newcommand{\cP}{{\mathfrak{P}}}
\newcommand{\cs}{{\mathfrak{s}}}

\newcommand{\cI}{{\mathfrak{I}}}

\newcommand{\cN}{{\mathfrak{N}}}
\newcommand{\cR}{{\mathfrak{R}}}
\newcommand{\fr}{{\mathfrak{r}}}

\allowdisplaybreaks

\begin{document}
\title[Nonlinear Schr\"odinger Equations With Quasi-Periodic Initial Data II.]{Existence, Uniqueness and Asymptotic Dynamics of Nonlinear Schr\"odinger Equations With Quasi-Periodic Initial Data: \\II. The Derivative NLS
}

\dedicatory{Dedicated to the memory of Thomas Kappeler}

\author{David Damanik}
\address{\scriptsize (D. Damanik)~Department of Mathematics, Rice University, 6100 S. Main Street, Houston, Texas
77005-1892}
\email{damanik@rice.edu}
\thanks{The first author (D. Damanik) was supported by Simons Fellowship $\# 669836$ and NSF grants DMS--1700131 and DMS--2054752}

\author{Yong Li}
\address{\scriptsize  (Y. Li)~Institute of Mathematics, Jilin University, Changchun 130012, P.R. China. School of Mathematics and Statistics, Center for Mathematics and Interdisciplinary Sciences, Northeast Normal University, Changchun, Jilin 130024, P.R.China.}
\email{liyong@jlu.edu.cn}
\thanks{The second author (Y. Li) was supported in part by National Basic Research Program of China (2013CB834100), and NSFC (12071175).}

\author{Fei Xu}
\address{\scriptsize (F. Xu)~Institute of Mathematics, Jilin University, Changchun 130012, P.R. China.}
\email{stuxuf@outlook.com}
\thanks{The third author (F. Xu) is sincerely grateful for the invitation to give a remote talk at UCLA on January 23, 2024, where the existence and uniqueness results on the (derivative) NLS with quasi-periodic initial data were announced publicly.}
\thanks{The authors sincerely thank Maria Ntekoume (Concordia University) for sharing paper \url{https://arxiv.org/pdf/2112.04648.pdf} with us.}

\begin{abstract}
This is the second part of a two-paper series studying the nonlinear Schr\"odinger equation with quasi-periodic initial data. In this paper, we focus on the quasi-periodic Cauchy problem for the derivative nonlinear Schr\"odinger equation. Under the assumption that the Fourier coefficients of the initial data obey an exponential upper bound, we establish local existence of a solution that retains quasi-periodicity in space with a slightly weaker Fourier decay. Moreover, the solution is shown to be unique within this class of quasi-periodic functions. Also, we prove that, for the derivative nonlinear Schr\"odinger equation in a weakly nonlinear setting, within the time scale, as the small parameter of nonlinearity tends to zero, the nonlinear solution converges asymptotically to the linear solution in the sense of both sup-norm and analytic Sobolev-norm.

The proof proceeds via a consideration of an associated infinite system of coupled ordinary differential equations for the Fourier coefficients and an explicit combinatorial analysis for the Picard iteration with the help of Feynman diagrams and the power of $\ast^{[\cdot]}$ labelling the complex conjugate.
\end{abstract}
\maketitle
\tableofcontents

\section{Introduction and Main Results}\label{intro}

This is the second part of our two-paper series to study the quasi-periodic Cauchy problem for Schr\"odinger-type equations, and we refer to the first paper \cite{DLX24I} for a more detailed introduction to the series.

In this paper, we study the following derivative nonlinear Schr\"odinger equation (dNLS
for short)
\begin{align}\label{dnls}
\tag{dNLS}{\rm i}\partial_tu+\partial_{xx}u-{\rm i}\partial_x(|u|^{2}u)=0
\end{align}
with the quasi-periodic initial data
\begin{align}\label{id}
u(0,x)=\sum_{n\in\bZ^\nu}\cc(n)e^{{\rm i}\langle n\rangle x},
\end{align}
where $\langle n\rangle\triangleq n\cdot\omega:=\sum_{j=1}^\nu n_j\omega_j$ and $\omega$ is rationally independent; $u$ is a complex-valued field defined on the real line $\mathbb R$; see \cite{HBRMV2022,RMV23APDE}. Here and below, $\partial_t$ and $\partial_x$ indicate time and spatial derivatives, respectively.

What we are interested in is to study the existence, uniqueness and asymptotic dynamics of spatially quasi-periodic solutions with the same frequency vector as the initial data to the above quasi-periodic Cauchy problem \eqref{dnls}-\eqref{id}. That is, such a solution is defined by Fourier series
\begin{align}\label{ses}
u(t,x)=\sum_{n\in\bZ^\nu}\cc(t,n) e^{{\rm i}\langle n\rangle x}, \quad x\in\mathbb R.
\end{align}

From the point of physics, the derivative NLS was first derived by plasma physicists in \cite{1976,19762} for studying the one-dimensional compressible magnetohydrodynamic equation in the presence of the Hall effect and describing the propagation of circular polarized Alfv\'en wave in magnetized plasmas with a constant magnetic field. In addition, the derivative NLS also depicts the phenomena of ultrashort optical pulses; see \cite{MMW07}. It should be emphasized that the derivative nonlinearity becomes physically important in the propagation of short pulses while negligible in many experimental scenarios; see \cite{HO92PD,HBRMV2022}. Some more details on the physical applications of the derivative NLS can be found in  \cite{SSBOOK,CLPS,JLPS,PT,BP22IM} and the references therein.

From the point of mathematics, \eqref{dnls} is well-known to be {\em completely integrable} and has all the properties of equations which are exactly solved by the inverse scattering technique.

(i)~It has an infinite number of conservation laws including the conservation of mass, momentum and energy:
\begin{align*}
M(u)&:=\int_{\mathbb R}|u|^2{\rm d}x,\\
H(u)&:={\text{Im}}\int_{\mathbb R}\overline{u}u_x{\rm d}x+\frac{1}{2}|u|^4{\rm d}x,\\
E(u)&:=\int_{\mathbb R}|u_x|^2+\frac{3}{2}{\text{Im}}(|u|^2u\overline{u_x})+\frac{1}{2}|u|^6{\rm d}x.
\end{align*}
The middle one here serves as the Hamiltonian for \eqref{dnls} since it generates the dynamics of \eqref{dnls} via the following Possion structure
\[\{F,G\}:=\int\frac{\delta F}{\delta u}\partial_x\frac{\delta G}{\delta\overline{u}}+\frac{\delta F}{\delta \overline{u}}\partial_x\frac{\delta G}{\delta u}{\rm d}x;\]
see \cite{HBRMV2022,RMV23APDE}.

(ii) It enjoys a Lax pair structure $(\mathcal U,\Upsilon)$, where
\begin{align*}
\mathcal{U}(\lambda)= & -{\rm i} \sigma_3\left(\lambda^2+{\rm i} \lambda U\right), \quad U=\left(\begin{array}{ll}
0 & u \\
\bar{u} & 0
\end{array}\right), \\
\Upsilon(\lambda)= & -{\rm i}\left(2\lambda^4-\lambda^2|u|^2\right)\sigma_3+\left(\begin{array}{cc}
0 & 2 \lambda^3 u-\lambda|u|^2 u+{\rm i}\lambda u_x \\
-2 \lambda^3 \bar{u}+\lambda|u|^2 \bar{u}+{\rm i}\lambda \overline{u_x} & 0
\end{array}\right),
\end{align*}
$\lambda\in\mathbb C$ is a $(t,x)$-independent spectral parameter, $\sigma_3$ is the Pauli matrix given by
$\sigma_3=\left(\begin{array}{cc}
1 & 0 \\
0 & -1
\end{array}\right)$.
This means that $u$ solves \eqref{dnls} if and only if we have the following zero curvature condition \[\frac{\partial \mathcal{U}}{\partial t}-\frac{\partial \Upsilon}{\partial x}+[\mathcal{U}, \Upsilon]=0;\]
see \cite{KN78JMP,L83,L89TAMS,BP22IM} and also \cite{JLPS2018CMP,JLPS2018CPDE} for its gauge-equivalent form. These give rise to the soliton phenomena; see \cite{HO92PD,JLPS2018CMP}.

What's more, compared with the standard nonlinear Schr\"odinger equation with power-law nonlinearity, there is a conspicuous derivative operator ${\rm i}\partial_x$ acting on the classical cubic nonlinear term. This leads to the facts that (i)~\eqref{dnls} does not admit a focusing-defocusing dichotomy (the nonlinearity can be reversed by simply replacing $x\mapsto-x$), and it does not inherit the Galilean symmetry of the linear Schr\"odinger equation; see \cite{HO92PD,NORBS12JEMS,HBRMV2022,BP22IM}. (ii)~\eqref{dnls} has a special $L^2$-invariance property compared with the cubic nonlinear Schr\"odinger equation, i.e., for $\lambda>0$, the rescaled function $u_\lambda(t,x)=\lambda^{q} u(\lambda^2t,\lambda x)$ solves \eqref{dnls} if and only if $q=1/2$. This scale parameter is like in the case of the quintic NLS, rather than the cubic NLS (recall that $q=1$ is the $L^2$-invariant value for the cubic NLS; see \cite{T06}).


As is well known, a lot of methods and techniques are available for the Cauchy problem of the nonlinear Schr\"odinger equation with power-law nonlinearity. However some of the them face limitations when directly applied to \eqref{dnls}, especially with quasi-periodic initial data, primarily due to challenges arising from the presence of a space derivative ${\rm i}\partial_x$ acting on the classical cubic nonlinear term. This is the so-called {\em derivative loss phenomenon}; see \cite{HO92PD,O96IUMJ}.
In our two-paper series, the combinatorial analysis method has been demonstrated to be applicable, at least for establishing local results, not only to the nonlinear Schr\"odinger equation (see the first paper \cite{DLX24I}) but also to \eqref{dnls}.

More specifically, in the quasi-periodic setting, the differential operator ${\rm i}\partial_x$ has coordinates $\{-\langle n\rangle\}$ under the quasi-periodic Fourier basis $\{e^{{\rm i}\langle n\rangle x}\}$. Due to encountering derivative loss, we have to tackle {\em the alternating discrete convolution of higher dimensions} in which the total distance $n$ will branch into {higher-dimensional} variable $n^{(k)}$ per the basic dual lattice $\bZ^\nu$ with the splitting condition $n=\cc\ca\cs(n^{(k)})$, serving as a growth factor.
It is much more difficult to control the growth of $|\cc\ca\cs(n^{(k)})|$ when trying to apply some standard tools; let us compare the case of $H^s(\mathbb T)$ \cite{NORBS12JEMS,IMDS24MA}, the case of $H^s(\mathbb R)$ \cite{CKSTT2001SIAMMA,CKSTT2002SIAMMA,LPS2018AIHP,JLPS2018CMP,JLPS2018CPDE,JLPS2020APDE,HBRMV2022,BP22IM}, the case of $H^s(\mathbb R^+)$ \cite{L08PD,L11PD,AL17N,EGT18}, and the case of real-valued PDE \cite{DG2017JAMS,MR4751185,DLX22AR}. In order to overcome this difficulty, in this two-paper series, we develop an explicit combinatorial method applied in \cite{DG2017JAMS,MR4751185,DLX22AR} by introducing {\em Feynman diagrams} and proposing the {\em power of $\ast^{[\cdot]}$} to label the complex conjugate (we call it the {\bf pcc} label method in the first paper \cite{DLX24I}) to analyze the complicated Picard iteration.

To this end, we have to demand a stronger decay condition (i.e., the so-called exponential decay condition) on the initial Fourier data to make the growth of $|\cc\ca\cs(n^{(k)})|$ controllable as $k$ tends to infinity, compared with our treatment of the standard NLS in the first paper \cite{DLX24I}.
Here we say that the initial Fourier data $\cc$ satisfies the exponential decay condition if there exists a pair $(B,\kappa)\in(0,\infty)\times(0,1]$ such that
\begin{align}\label{ed}
|\cc(n)|\leq B^{1/2}e^{-\kappa|n|},\quad\forall n\in\bZ^\nu.
\end{align}

Under the exponential decay condition \eqref{ed}, applying the combinatorial analysis method together with Feynman diagrams and the {\bf pcc} label method (\autoref{pcc}),
we obtain the following main results for the quasi-periodic Cauchy problem \eqref{dnls}-\eqref{id}.
\begin{theo}[dNLS]\label{dnlsth}
If the initial Fourier data $\cc$ is $\mathtt \kappa$-exponentially decaying in the sense of \eqref{ed}, then the following statements hold:
\begin{enumerate}
  \item (Existence)~The quasi-periodic Cauchy problem \eqref{dnls}-\eqref{id} has a spatially quasi-periodic solution \eqref{ses} with the same frequency vector as the initial data (i.e., it retains the same spatial quasi-periodicity) defined on $[0,t_1)\times\bR$, where $t_1=\min\{t_2,t_3\}$ (see \eqref{t2} and \eqref{t3} for the definition of $t_2$ and $t_3$ respectively).
  \item (Decay and smoothness)~The spatially quasi-periodic solution \eqref{ses} is, uniformly in $t$, $\kappa/2$-exponentially decaying (with a slightly worse decay rate), that is, \begin{align}\label{sdecay}
      |\cc(t,n)|\lesssim e^{-\kappa/2},\quad \forall (t,n)\in[0,t_1)\times\mathbb Z^\nu.
       \end{align}
       Hence this solution is classical in time and analytic in space (see \cite[Lemma 4.5]{DLX24I}).
  \item (Uniqueness)~The spatially quasi-periodic solution \eqref{ses} with exponentially decaying Fourier coefficients  is unique on $[0,t_4)\times\bR$ (see \eqref{t4} for the definition of $t_4$).
  \item (Asymptotic dynamics)~Consider the quasi-periodic Cauchy problem for the derivative NLS \eqref{dnls} with small nonlinearity, that is, with $0<|\epsilon|\ll1$,
\begin{align}
\tag{$\epsilon$-dNLS}{\rm i}\partial_tu+\partial_{xx}u-{\rm i}\epsilon\partial_x(|u|^{2}u)=0.
\end{align}
Then for $t=|\epsilon|^{-1+\eta}<|\epsilon|^{-1}$ with $0<\eta\ll1$, as $\epsilon\rightarrow0$, we have
\begin{align}
L^\infty\text{-asymptoticity}:\quad&\|u^\epsilon(t)-u_{\text{linear}}(t)\|_{L_x^\infty(\mathbb R)}\rightarrow 0;\\
\text{Analytic Sobolev asymptoticity}:\quad&
\|u^\epsilon(t)-u_{\text{linear}}(t)\|_{\mathcal H_x^{\varrho}(\mathbb R)}\rightarrow 0, \quad\left(0<\frac{\kappa}{4}-4\varrho\leq1\right),
\end{align}
where $\|f\|_{L_x^{\infty}(\mathbb R)}=\max_{x\in\mathbb R}|f(x)|$ and $\|f\|_{\mathcal H_x^{\varrho}(\mathbb R)}
=\|e^{\varrho|n|}\hat{f}(\langle n\rangle)\|_{\ell_{n}^2(\mathbb Z^\nu)}$ with $\varrho>0$.
\end{enumerate}
\end{theo}

\begin{rema}
(a)~The existence and uniqueness results were publicly announced in January 2024 in a talk given by Fei Xu in the Analysis and PDE Seminar at UCLA; see \cite{F2024CULA}. About a month after this seminar talk, Hagen Papenburg posted the preprint \cite{P} on the arXiv, in which he develops an alternative approach to studying dispersive PDEs with quasi-periodic initial data.

(b)~The existence and uniqueness results are about {\em large data local solutions} for the one-dimensional derivative NLS equation with quasi-periodic initial data, where a smallness condition is not required.

(c)~The asymptotic dynamics result requires a smallness condition on the nonlinearity, which is essentially a {\em small data quasi-periodic Cauchy problem}. This is motivated by \cite{GGKS}. However we do not know what will happen for the critical case $t\sim|\epsilon|^{-1}$ and the supercritical case $t>|\epsilon|^{-1}$.
\end{rema}

\section{Proof of Theorem \ref{dnlsth}}\label{secdnls}

\subsection{Preliminaries}\label{pcc}

Before giving the proof, we collect some concepts and notations from the first paper \cite{DLX24I}, especially for the case $p=1~(P=3)$.

\subsubsection{Power of $\ast^{[\cdot]}$ for the Complex Conjugate}
The first important concept is the so-called {power of $\ast^{[\cdot]}$ for the complex conjugate} ({\bf pcc} for short). Specifically, we use $z^{\ast^0}$ and $z^{\ast^1}$ to stand for the complex number $z$ and its complex conjugate $\bar z$ respectively. With this notation at hand, the $m$-times complex conjugate of $z$ is $z^{\ast^{[m]}}$, where $\{0,1\}\ni[m]\equiv m~(\text{mod}~2)$. In addition, as to $[\cdot]$, we have the following operation properties:
\begin{itemize}
  \item $ \overline{z^{\ast^{[m]}}}=z^{\ast^{[m+1]}}$ for all $m\in\bN$ and $z\in\bC$;
  \item  $ \left(z^{\ast^{[m]}}\right)^{\ast^{[m^\prime]}}=
 z^{\ast^{[m+m^\prime]}}$ for all $m,m^\prime\in\bN$ and $z\in\bC$;
 \item $ (z_1+z_2)^{\ast^{[m]}}=z_1^{\ast^{[m]}}+z_2^{\ast^{[m]}}$ for all $m\in\bN$ and $z_1,z_2\in\bC$;
  \item $ (z_1z_2)^{\ast^{[m]}}=z_1^{\ast^{[m]}}\cdot z_2^{\ast^{[m]}}$ for all $m\in\bN$ and $z_1,z_2\in\bC$.
\end{itemize}
\subsubsection{Combinatorial Structure and Some Basic Concepts}
Next we introduce some combinatorial concepts and notations that are related to the cubic nonlinearity $u\bar uu$.

The {\bf branch set} $\Gamma^{(k)}$ is defined by letting
\begin{align}\label{g}
\Gamma^{(k)}=
\begin{cases}
\{0,1\},&k=1;\\
\{0\}\cup(\Gamma^{(k-1)})^{3},&k\geq2.
\end{cases}
\end{align}
It is used to label or follow every term w.r.t. the initial data in the Picard iteration.

The {\bf first counting function} $\sigma$ ($2\sigma$
indeed) acting on the branch set is defined by letting
\begin{align}\label{s}
\sigma(\gamma^{(k)})=
\begin{cases}
\frac{1}{2},&\gamma^{(k)}=0\in\Gamma^{(k)},k\geq1;\\
\frac{3}{2},&\gamma^{(1)}=1\in\Gamma^{(1)};\\
\sum_{j=1}^{3}\sigma(\gamma_j^{(k-1)}),&\gamma^{(k)}=(\gamma_j^{(k-1)})_{1\leq j\leq 3}\in(\Gamma^{(k-1)})^{3},k\geq2.
\end{cases}
\end{align}
It depicts the degree/multiplicity of the nonlinearity in the sense of the Picard iteration, that is, the number of the initial Fourier data on each branch in the Picard iteration.

The {\bf second counting function} $\ell$ ($2\ell$ indeed) acting on the branch set is defined by letting
\begin{align}\label{ee}
\ell(\gamma^{(k)})=
\begin{cases}
0,&\gamma^{(k)}=0\in\Gamma^{(k)},k\geq1;\\
1,&\gamma^{(1)}=1\in\Gamma^{(1)};\\
1+\sum_{j=1}^{3}\ell(\gamma_j^{(k-1)}),&\gamma^{(k)}=(\gamma_j^{(k-1)})_{1\leq j\leq 3}\in(\Gamma^{(k-1)})^{3},k\geq2.
\end{cases}
\end{align}
It stands for the number of time integrations.

With these intuitions in mind, by induction, we have that, for all $k\geq1$,
$2\sigma(\gamma^{(k)})$ is odd, $2\ell(\gamma^{(k)})$ is even, and
\begin{align}\label{sil}
\sigma(\gamma^{(k)})=\ell(\gamma^{(k)})+1/2.
\end{align} 

The {\bf combinatorial lattice space} $\cN^{(k,\gamma^{(k)})}$ {originated from} $\bZ^\nu$ on each branch is defined by letting
\begin{align}\label{n}
&\cN^{(k,\gamma^{(k)})}=
\begin{cases}
\bZ^\nu,&\gamma^{(k)}=0\in\Gamma^{(k)},k\geq1;\\
(\bZ^\nu)^{3},&\gamma^{(1)}=1\in\Gamma^{(1)};\\
\prod_{j=1}^{3}\cN^{(k-1,\gamma_j^{(k-1)})},&\gamma^{(k)}=(\gamma_j^{(k-1)})_{1\leq j\leq 3}\in(\Gamma^{(k-1)})^{3},k\geq2.
\end{cases}
\end{align}
Let dim$_{\bZ^{\nu}}\cN^{(k,\gamma^{(k)})}$ be the number of components in $\cN^{(k,\gamma^{(k)})}$ per $\bZ^\nu$. Then we have
\begin{align}\label{dimsi}
\dim_{\bZ^\nu}\cN^{(k,\gamma^{(k)})}=
2\sigma(\gamma^{(k)}),\quad\forall k\geq1.
\end{align}
Hence it is reasonable to set $n^{(k)}=(m_j)_{1\leq j\leq 2\sigma(\gamma^{(k)})}$, where $m_j\in\bZ^\nu$ for all $j=1,\cdots,3$.

The {\bf combinatorial alternating sums}, denoted by $\cc\ca\cs(n^{(k)})$, of $n^{(k)}\in\cN^{(k,\gamma^{(k)})}$, is defined by letting
\begin{align}\label{as1}
&\cc\ca\cs(n^{(k)})=
\begin{cases}
n^{(k)},&\gamma^{(k)}=0\in\Gamma^{(k)},n^{(k)}\in\cN^{(k,0)}, k\geq1;\\
\sum_{j=1}^{3}(-1)^{j-1}n_j,&\gamma^{(1)}=1\in\Gamma^{(1)},n^{(1)}=(n_j)_{1\leq j\leq 3}\in(\bZ^\nu)^{3};\\
\sum_{j=1}^{3}(-1)^{j-1}\cc\ca\cs(n_j^{(k-1)}),&\gamma^{(k)}=(\gamma_j^{(k-1)})_{1\leq j\leq 3}\in(\Gamma^{(k-1)})^{3},\\
&n^{(k)}=(n_j^{(k-1)})_{1\leq j\leq {3}}\in\prod_{j=1}^{3}\cN^{(k-1,\gamma_j^{(k-1)})},\\
& k\geq2.
\end{cases}
\end{align}
It follows from induction that
\begin{align}\label{as2}
\cc\ca\cs(n^{(k)})=\sum_{j=1}^{2\sigma(\gamma^{(k)})}(-1)^{j-1}m_j;
\end{align}
see also the Feynman diagram in the first paper \cite{DLX24I}.

\subsection{Main Steps of Proof}

The proof of Theorem \ref{dnlsth} is divided into the following subsubsections.

\subsubsection{Infinite-Dimensional ODEs, Picard Iteration and Combinatorial Tree}

In this subsubsection we reduce the quasi-periodic Cauchy problem for \eqref{dnls} to an associated infinite system of coupled ordinary differential equations for Fourier coefficients. Then we define the Picard iteration to approximate them.

{Since the frequency vector $\omega$ is rationally independent (this implies that $\{e^{i\langle n\rangle x}\}$ is orthogonal w.r.t. the inner product in the mean sense)}, plugging \eqref{ses} into \eqref{dnls} yields the following nonlinear infinite system of coupled ODEs
\begin{align}
\frac{{\rm d}}{{\rm d}t}\cc(t,n)+{\rm i}\langle n\rangle^2\cc(t,n)={\rm i}\langle n\rangle\sum_{\substack{n_1,n_2,n_3\in\bZ^\nu\\n_1-n_2+n_3=n}}\prod_{j=1}^{3}\{\cc(t,n_j)\}^{\ast^{[j-1]}},\quad\forall n\in\bZ^\nu.
\end{align}
{It follows from Duhamel's principle} that they are equivalent to the following integral form
\begin{align}
\cc(t,n)=e^{-{\rm i}\langle n\rangle^2t}\cc(n)+{\rm i}\langle n\rangle\int_0^te^{-{\rm i}\langle n\rangle^2(t-s)}\sum_{\substack{n_1,n_2,n_3\in\bZ^\nu\\n_1-n_2+n_3=n}}\prod_{j=1}^{3}\{\cc(s,n_j)\}^{\ast^{[j-1]}}{\rm d}s,\quad\forall n\in\bZ^\nu.
\end{align}
To approximate the unknown Fourier coefficient $\cc(t,n)$, define the Picard sequence $\{\cc_k(t,n)\}_{k\geq0}$ by choosing the linear solution as the initial guess
\begin{align}\label{c0}
\cc_0(t,n)=e^{-{\rm i}\langle n\rangle^2t}c(n)
\end{align}
and successively defining ${{\cc}_k}_{k\geq1}$ as follows:
\begin{align}\label{gsj}
\cc_k(t,n)=\cc_0(t,n)+{\rm i}\langle n\rangle\int_0^te^{-{\rm i}\langle n\rangle^2(t-s)}\sum_{\substack{n_1,n_2,n_3\in\bZ^\nu\\n_1-n_2+n_3=n}}\prod_{j=1}^{3}\{\cc_{k-1}(s,n_j)\}^{\ast^{[j-1]}}{\rm d}s,\quad\forall k\geq1,
\end{align}

This iteration is complicated; see Remark 3.2 in \cite{DLX24I}. To overcome this difficulty, an explicit combinatorial method with Feynman diagram and {\bf pcc} label method is applied to analyze the Picard iteration.

%
{Define $\tC,\tI$ and $\tF$ as follows:
\begin{align*}
&\tC^{(k,\gamma^{(k)})}(n^{(k)})=\\
&
\begin{cases}
c\left(\cc\ca\cs(n^{(k)})\right),&0=\gamma^{(k)}\in\Gamma^{(k)}, n^{(k)}\in\cN^{(k,0)},k\geq1;\\
\prod_{j=1}^{3}\left\{c(n_j)\right\}^{\ast^{[j-1]}},&1=\gamma^{(1)}\in\Gamma^{(1)}, n^{(1)}\in\cN^{(1,1)};\\
\prod_{j=1}^{3}\left\{\tC^{(k-1,\gamma_j^{(k-1)})}(n^{(k-1)}_j)\right\}^{\ast^{[j-1]}},
&\gamma^{(k)}=(\gamma_j^{(k-1)})_{1\leq j\leq3}\in(\Gamma^{(k-1)})^3, \\
&n^{(k)}=(n_j^{(k-1)})_{1\leq j\leq3}\in\prod_{j=1}^3\cN^{(k-1,\gamma_j^{(k-1)})},\\
&k\geq2.
\end{cases}\\
&\tI^{(k,\gamma^{(k)})}(t,n^{(k)})=\\
&
\begin{cases}
e^{-{\rm i}\left\langle\cc\ca\cs(n^{(k)})\right\rangle^2t},&0=\gamma^{(k)}\in\Gamma^{(k)}, \\
&n^{(k)}\in\cN^{(k,0)},\\
&k\geq1;\\
\int_0^te^{-{\rm i}\left\langle\cc\ca\cs(n^{(1)})\right\rangle^2(t-s)}\prod_{j=1}^{3}\left\{e^{-{\rm i}\langle n_j\rangle^2s}\right\}^{\ast^{[j-1]}}{\rm d}s,&1=\gamma^{(1)}\in\Gamma^{(1)}, \\
&n^{(1)}\in\cN^{(1,1)};\\
\int_0^te^{-{\rm i}\left\langle\cc\ca\cs(n^{(k)})\right\rangle^2(t-s)}\prod_{j=1}^{3}\left\{\cI^{(k-1,\gamma_j^{(k-1)})}(s,n_j^{(k-1)})\right\}^{\ast^{[j-1]}}{\rm d}s,\\[1mm]
&\hspace{-5cm}\gamma^{(k)}=(\gamma_j^{(k-1)})_{1\leq j\leq3}\in(\Gamma^{(k-1)})^3, \\
&\hspace{-5cm}n^{(k)}=(n_j^{(k-1)})_{1\leq j\leq3}\in\prod_{j=1}^3\cN^{(k-1,\gamma_j^{(k-1)})},k\geq2.
\end{cases}\\
&\tF^{(k,\gamma^{(k)})}(n^{(k)})=\\
&
\begin{cases}
1,&0=\gamma^{(k)}\in\Gamma^{(k)}, n^{(k)}\in\cN^{(k,0)},\\
&k\geq1;\\
{\rm i}\langle\cc\ca\cs(n^{(1)})\rangle,&1=\gamma^{(1)}\in\Gamma^{(1)}, n^{(1)}\in\cN^{(1,1)};\\
{\rm i}\langle\cc\ca\cs(n^{(k)})\rangle\prod_{j=1}^3\{\tF^{(k-1,\gamma_j^{(k-1)})}(n_j^{(k-1)})\}^{\ast^{j-1}},
&\gamma^{(k)}=(\gamma_j^{(k-1)})_{1\leq j\leq3}\in(\Gamma^{(k-1)})^3, \\
&n^{(k)}=(n_j^{(k-1)})_{1\leq j\leq3}\\
&\in\prod_{j=1}^3\cN^{(k-1,\gamma_j^{(k-1)})},\\
&k\geq2.
\end{cases}
\end{align*}
}

%
%
%

By induction, the Picard sequence enjoys the following combinatorial tree form \eqref{ckk}.
\begin{lemm}\label{lemcc}
For all $k\geq1$,
\begin{align}\label{ckk}
\cc_k(t,n)=\sum_{\gamma^{(k)}\in\Gamma^{(k)}}\sum_{\substack{n^{(k)}\in\cN^{(k,\gamma^{(k)})}\\\cc\ca\cs(n^{(k)})=n}}
\tC^{(k,\gamma^{(k)})}(n^{(k)})\tI^{(k,\gamma^{(k)})}(t,n^{(k)})
\tF^{(k,\gamma^{(k)})}(n^{(k)}).
\end{align}
\end{lemm}
{
\begin{proof}
Notice that for all $\gamma^{(k)}=0\in\Gamma^{(k)}$, where $k\geq1$, we have
\[\cc_0(t,n)=\sum_{\substack{n^{(k)}\in\cN^{(k,0)}\\\cc\ca\cs(n^{(k)})=n}}\tC^{(k,0)}(n^{(1)})\tI^{(k,0)}(t,n^{(k)})\tF^{(k,0)}(n^{(k)}).\]

For $k=1$, substituting \eqref{c0} into \eqref{gsj}, with the properties of $\ast^{[\cdot]}$, yields that
\begin{align*}
\cc_1(t,n)-\cc_0(t,n)
&=\sum_{\substack{n_1,n_2,n_3\in\bZ^\nu\\n_1-n_2+n_3=n}}{\rm i}\langle n\rangle\prod_{j=1}^3\{\cc(n_j)\}^{\ast^{[j-1]}}\int_0^te^{-{\rm i}\langle n\rangle^2(t-s)}\prod_{j=1}^3\{e^{-{\rm i}\langle n_j\rangle^2}\}^{\ast^{[j-1]}}{\rm d}s\\
&=\sum_{\substack{n^{(1)}\in\cN^{(1,1)}\\\cc\ca\cs(n^{(1)})=n}}\tC^{(1,1)}(n^{(1)})\tI^{(1,1)}(t,n^{(1)})\tF^{(1,1)}(n^{(1)}).
\end{align*}
Hence \eqref{ckk} holds for $k=1$.

Let $k\geq2$. Assume that \eqref{ckk} is true for all $1<k^\prime<k$. For $k$, it follows from \eqref{gsj} and the properties of the label of complex conjugate that
\begin{align*}
&\cc_k(t,n)-\cc_0(t,n)\\
=&{\rm i}\langle n\rangle\int_0^te^{-{\rm i}\langle n\rangle^2(t-s)}\sum_{\substack{n_1,n_2,n_3\in\bZ^\nu\\n_1-n_2+n_3=n}}\prod_{j=1}^3\\
&\Bigg\{\sum_{\gamma_j^{(k-1)}\in\Gamma^{(k-1)}}\sum_{\substack{n_j^{(k-1)}\in\cN^{(k-1,\gamma_j^{(k-1)})}\\\cc\ca\cs(n_j^{(k-1)})=n_j}}
\tC^{(k-1,\gamma_j^{(k-1)})}(n_j^{(k-1)})\tI^{(k-1,\gamma_j^{(k-1)})}(s,n_j^{(k-1)})
\\
&\tF^{(k-1,\gamma_j^{(k-1)})}(n_j^{(k-1)})\Bigg\}^{\ast^{[j-1]}}
\\
=&{\rm i}\langle n\rangle\int_0^te^{-{\rm i}\langle n\rangle^2(t-s)}\sum_{\substack{n_1,n_2,n_3\in\bZ^\nu\\n_1-n_2+n_3=n}}\prod_{j=1}^3\sum_{\gamma_j^{(k-1)}\in\Gamma^{(k-1)}}
\sum_{\substack{n_j^{(k-1)}\in\cN^{(k-1,\gamma_j^{(k-1)})}\\\cc\ca\cs(n_j^{(k-1)})=n_j}}\\
&\{\tC^{(k-1,\gamma_j^{(k-1)})}(n_j^{(k-1)})\}^{\ast^{[j-1]}}
\{\tI^{(k-1,\gamma_j^{(k-1)})}(s,n_j^{(k-1)})\}^{\ast^{[j-1]}}
\{\tF^{(k-1,\gamma_j^{(k-1)})}(n_j^{(k-1)})\}^{\ast^{[j-1]}}\\
=&{\rm i}\langle n\rangle\int_0^te^{-{\rm i}\langle n\rangle^2(t-s)}\sum_{\substack{n_1,n_2,n_3\in\bZ^\nu\\n_1-n_2+n_3=n}}
\sum_{\substack{\gamma_j^{(k-1)}\in\Gamma^{(k-1)}\\j=1,2,3}}
\sum_{\substack{n_j^{(k-1)}\in\cN^{(k-1,\gamma_j^{(k-1)})}\\\cc\ca\cs(n_j^{(k-1)})=n_j\\j=1,2,3}}\\
&\prod_{j=1}^3\{\tC^{(k-1,\gamma_j^{(k-1)})}(n_j^{(k-1)})\}^{\ast^{[j-1]}}\prod_{j=1}^3
\{\tI^{(k-1,\gamma_j^{(k-1)})}(s,n_j^{(k-1)})\}^{\ast^{[j-1]}}\\
&\prod_{j=1}^3\{\tF^{(k-1,\gamma_j^{(k-1)})}(n_j^{(k-1)})\}^{\ast^{[j-1]}}\\
=&\sum_{\substack{\gamma_j^{(k-1)}\in\Gamma^{(k-1)}\\j=1,2,3}}\sum_{\substack{n_1,n_2,n_3\in\bZ^\nu\\n=n_1-n_2+n_3}}
\sum_{\substack{n_j=\cc\ca\cs(n_j^{(k-1)})\\n_j^{(k-1)}\in\Gamma^{(k-1)}\\j=1,2,3}}{\rm i}\langle n\rangle\prod_{j=1}^3\{\tF^{(k-1,\gamma_j^{(k-1)})}(n_j^{(k-1)})\}^{\ast^{[j-1]}}\\
&\prod_{j=1}^3\{\tC^{(k-1,\gamma_j^{(k-1)})}(n_j^{(k-1)})\}^{\ast^{[j-1]}}
\int_0^te^{-{\rm i}\langle n\rangle^2}\prod_{j=1}^3\{\tI^{(k-1,\gamma_j^{(k-1)})}(s,n_j^{(k-1)})\}^{\ast^{[j-1]}}{\rm d}s\\
=&\sum_{\gamma^{(k)}\in(\Gamma^{(k-1)})^3}\sum_{\substack{n^{(k)}\in\cN^{(k,\gamma^{(k)})}\\\cc\ca\cs(n^{(k)})=n}}
\tC^{(k,\gamma^{(k)})}(n^{(k)})\tI^{(k,\gamma^{(k)})}(t,n^{(k)})\tF^{(k,\gamma^{(k)})}(n^{(k)}).
\end{align*}
This shows that \eqref{ckk} holds for $k$, and for all $k\in\bN$ by induction. The proof of Lemma \ref{lemcc} is completed.
\end{proof}
}

\subsubsection{The Picard Sequence is Exponentially Decaying}

In this subsubsection we will prove that the Picard sequence is $\kappa/2$-exponentially decaying with a slightly worse decay rate compared with the initial Fourier data.

First we estimate $\tC,\tI$ and $\tF$ in \eqref{ccc} independently.

For $\tC$, it follows from induction and the {\bf pcc} label method that it has a good combinatorial structure \eqref{ccc}.
\begin{lemm}\label{cscscs}
For all $k\geq 1$, set $n^{(k)}=(m_j)_{1\leq j\leq2\sigma(\gamma^{(k)})}\in\cN^{(k,\gamma^{(k)})}$, we have
\begin{align}\label{ccc}
\tC^{(k,\gamma^{k})}(n^{(k)})=\prod_{j=1}^{2\sigma(\gamma^{(k)})}\left\{\cc(m_j)\right\}^{^{\ast^{[j-1]}}}.
\end{align}
\end{lemm}
\begin{proof}
See the first paper \cite[Lemma 3.4]{DLX24I}.
\end{proof}

With the help of Lemma \ref{cscscs}, under the exponential decay condition \eqref{ed} of the initial Fourier data, an estimate of $\tC$ is obtained as follows.
\begin{lemm}\label{lsds}
If the initial Fourier data $\cc$ satisfies the exponential decay condition \eqref{ed}, then
\begin{align}
|\tC^{(k,\gamma^{(k)})}(n^{(k)})|\leq B^{\sigma(\gamma^{(k)})}e^{-\kappa|n^{(k)}|},\quad\forall k\geq1.
\end{align}
\end{lemm}
\begin{proof}
It follows from Lemma \ref{cscscs} and the exponential decay condition \eqref{ed} that
\begin{align*}
|\tC^{(k,\gamma^{k})}(n^{(k)})|&=\prod_{j=1}^{2\sigma(\gamma^{(k)})}|\left\{\cc(m_j)\right\}^{^{\ast^{[j-1]}}}|\\
&\leq\prod_{j=1}^{2\sigma(\gamma^{(k)})}B^{1/2}e^{-\kappa|m_j|}\\
&=B^{\sigma(\gamma^{(k)})}e^{-\kappa|n^{(k)}|}.
\end{align*}
Here $|n^{(k)}|=\sum_{j=1}^{2\sigma(\gamma^{(k)})}|m_j|$. This completes the proof of Lemma \ref{lsds}.
\end{proof}

For $\tI$, we have an estimate by induction as follows.
\begin{lemm}
For all $k\geq1$,
\begin{align}
|\tI^{(k,\gamma^{(k)})}(t,n^{(k)})|\leq\frac{t^{\ell(\gamma^{(k)})}}{\tD(\gamma^{(k)})},
\end{align}
where $\ell(\gamma^{(k)})=\sigma(\gamma^{(k)})-\frac{1}{2}$ on each branch $\gamma^{(k)}$, and $\tD$ is defined as follows:
\begin{align}\label{pi}
\tD(\gamma^{(k)})=
\begin{cases}
1,&\gamma^{(k)}=0\in\Gamma^{(k)},k\geq1;\\
1,&\gamma^{(1)}=1\in\Gamma^{(1)};\\
\ell(\gamma^{(k)})\prod_{j=1}^{3}\tD(\gamma_j^{(k-1)}),&\gamma^{(k)}=(\gamma_j^{(k-1)})_{1\leq j\leq 3}\in(\Gamma^{(k-1)})^{3},k\geq2.
\end{cases}
\end{align}
\end{lemm}
\begin{proof}
See the first paper \cite[Lemma 3.6]{DLX24I}.
\end{proof}
Compared with \cite[Lemma 3.7]{DLX24I}, we need to re-estimate the term $\tF$ in \eqref{dnls} due to the growth factor $|\cc\ca\cs(n^{(k)})|$, which is also different from $|\mu(n^{(k)})|$ appearing in \cite{DG2017JAMS}, and obtain the following result.
\begin{lemm}\label{alsp}
For all $k\geq1$, we have
\begin{align}
|\tF^{(k,\gamma^{(k)})}(n^{(k)})|
&\leq|\omega|^{\ell(\gamma^{(k)})}\sum_{\alpha=(\alpha_j)_{1\leq j\leq2\sigma(\gamma^{(k)})}\in\mathfrak R^{(k,\gamma^{(k)})}}\prod_{j}|(n^{(k)})_j|^{\alpha_j},\label{cie}
\end{align}
where
\begin{align*}
\mathfrak R^{(k,\gamma^{(k)})}&:=
\begin{cases}
\{0\in\mathbb Z\}, &\gamma^{(k)}=0\in\Gamma^{(k)}, k\geq1;\\
\{(1,0,0),(0,1,0),(0,0,1)\}, &\gamma^{(1)}=1\in\Gamma^{(1)};\\
\prod_{j=1}^{3}\mathfrak R^{(k-1,\gamma_j^{(k-1)})}+\ce^{(k,\gamma^{(k)})}, &\gamma^{(k)}=(\gamma_j^{(k-1)})_{1\leq j\leq3}\\
&\in\prod_{j=1}^{3}\mathfrak N^{(k-1,\gamma_j^{(k-1)})},k\geq2,
\end{cases}
\end{align*}
and
\[\ce^{(k,\gamma^{(k)})}:=\{\alpha\in\mathbb Z^{2\sigma(\gamma^{(k)})}: |\alpha|=1,\alpha_j\geq0\}.\]
\end{lemm}
\begin{proof}
Here we give a quick proof by induction, in which we do not introduce the symbol $\cP$ (see \cite[(2.15) and Lemma 2.8]{DG2017JAMS}) or the symbol $\cB$ (see \cite[(2.4) and Lemma 2.10]{MR4751185}).

The key ingredient is to control $|\cc\ca\cs(n^{(k)})|$ in a combinatorial manner of its components by introducing the new variable $\alpha$. It follows from \eqref{as2} that
\begin{align*}
\cc\ca\cs(n^{(k)})&=\sum_{j=1}^{2\sigma
(\gamma^{(k)})}(-1)^{j-1}(n^{(k)})_j\\
&\ll\sum_{j=1}^{2\sigma(\gamma^{(k)})}|(n^{(k)})_j|\\
&=\sum_{\alpha=(\alpha_j)_{1\leq j\leq2\sigma(\gamma^{(k)})}\in\ce^{(k,\gamma^{(k)})}}\prod_{j=1}^{2\sigma(\gamma^{(k)})}|(n^{(k)})_j|^{\alpha_j},
\end{align*}
where $n^{(k)}\in\cN^{(k,\gamma^{(k)})},\gamma^{(k)}\in\Gamma^{(k)}$ and $k\geq1$.

It is clear to see that \eqref{cie} holds for all $\gamma^{(k)}=0\in\Gamma^{(k)}$, where $k\geq1$, and $k=1, 1=\gamma^{(1)}\in\Gamma^{(1)},n^{(1)}\in\cN^{(1,1)}$. This shows that \eqref{cie} holds for $k=1$.

Let $k\geq2$. Assume that \eqref{cie} holds for all $1<k^\prime<k$.


For $k$, we need to consider only $\gamma^{(k)}=(\gamma_j^{(k-1)})_{1\leq j\leq3}\in\prod_{j=1}^3\Gamma^{(k-1)},n^{(k)}=(n_j^{(k-1)})_{1\leq j\leq3}\in\prod_{j=1}^3\cN^{(k-1,\gamma_j^{(k-1)})}$, it follows from the definition of $\tF$, induction hypothesis and \eqref{ee} that
\begin{align*}
&\tF^{(k,\gamma^{(k)})}(n^{(k)})\\
=&{\rm i}\langle\cc\ca\cs(n^{(k)})\rangle\prod_{j=1}^3{\tF^{(k-1,\gamma_j^{(k-1)})}}(n_j^{(k-1)})\\
\ll& |\omega|\sum_{j=1}^{2\sigma(\gamma^{(k)})}|(n^{(k)})_j|\prod_{j=1}^3|\omega|^{\ell(\gamma_j^{(k-1)})}
\sum_{\substack{\alpha^{(j)}=(\alpha^{(j)}_{i(j)})
_{1\leq i(j)\leq2\sigma(\gamma_j^{(k-1)})}\\\in\cR^{(k-1,\gamma_j^{(k-1)})}}}\prod_{i(j)=1}^{2\sigma(\gamma_j^{(k-1)})}|(n_j^{(k-1)})
_{i(j)}|^{\alpha_{i(j)}^{(j)}}
\\
=&|\omega|^{\ell(\gamma^{(k)})}\sum_{j=1}^{2\sigma(\gamma^{(k)})}|(n^{(k)})_j|\prod_{j=1}^3|\omega|^{\ell(\gamma_j^{(k-1)})}
\sum_{\alpha=(\alpha_{i(j)}^{(j)})_{\substack{1\leq j\leq 3; 1\leq i(j)\leq 2\sigma(\gamma_j^{(k-1)})\\\in\prod_{j=1}^3\cR^{(k-1,\gamma_j^{(k-1)})}}}}\prod_{j=1}^3|(n_{j}^{(k-1)})_{i(j)}|^{\alpha_{i(j)}^{(j)}}
\\
=&|\omega|^{\ell(\gamma^{(k)})}\sum_{\alpha=(\alpha_j)_{1\leq j\leq2\sigma(\gamma^{(k)})}\in\ce^{(k,\gamma^{(k)})}}\prod_{j=1}^{2\sigma(\gamma^{(k-1)})}|(n^{(k)})_j|^{\alpha_j}
\prod_{\substack{\alpha=(\alpha_j)_{1\leq j\leq2\sigma(\gamma^{(k)})}\\\in\prod_{j=1}^3\cR^{(k-1,\gamma_j^{(k-1)})}}}\prod_{j=1}^{2\sigma(\gamma^{(k)})}|(n^{(k)})_{j}|^{\alpha_j}
\\=&|\omega|^{\ell(\gamma^{(k)})}\prod_{\alpha=(\alpha_j)_{1\leq j\leq2\sigma(\gamma^{(k)})}\in\cR^{(k,\gamma^{(k)})}}\prod_{j=1}^{2\sigma(\gamma^{(k)})}|(n^{(k)})_{j}|^{\alpha_j}.
\end{align*}
%
Hence \eqref{cie} holds for $k$, and all $k\geq1$ by induction. This finishes the proof of Lemma \ref{alsp}.
\end{proof}

Putting everything together and using the equality \eqref{sil} yields the following estimates for the Picard sequence.

\begin{lemm}\label{less}
For all $k\geq1$, we have
\begin{align}\label{cksl}
&\cc_k(t,n)\nonumber\\
\ll& B^{1/2}\sum_{\gamma^{(k)}\in\Gamma^{(k)}}\frac{(B|\omega|t)^{\ell(\gamma^{(k)})}}{\tD(\gamma^{(k)})}
\sum_{\substack{\alpha=(\alpha_j)_{1\leq j\leq2\sigma(\gamma^{(k)})}\\\in\mathfrak R^{(k,\gamma^{(k)})}}}\sum_{\substack{n^{(k)}\in\cN^{(k,\gamma^{(k)})}\\\cc\ca\cs(n^{(k)})=n}}\prod_{j}|(n^{(k)})_j|^{\alpha_j}e^{-\kappa|n^{(k)}|}.
\end{align}
\end{lemm}
{
\begin{proof}
This is obtained by Lemma \ref{lemcc}, Lemmas \ref{lsds}-\ref{alsp}, and \eqref{sil}.
\end{proof}
}
To complete the estimate of the right hand side of \eqref{cksl}, we need the following Lemma \ref{4}.
\begin{lemm}\label{4}
For $0 < \kappa \leq 1$, we have
$$
\sum_{\substack{m=(m_1,\cdots,m_r)\in(\mathbb Z^{\nu})^r\\ \cc\ca\cs(m)=n}}\prod_{j=1}^{r}|m_j|^{\alpha_j}e^{-\kappa|m_j|}\leq e^{-\frac{\kappa}{2}|n|}(12\kappa^{-1})^{|\alpha|+\nu r}\prod_{j=1}^{r}\alpha_j!,
$$
where $\cc\ca\cs$ refers to the combinatorial alternating sums; see \eqref{as2}.
\end{lemm}

\begin{proof}
By Lemma \ref{3} and \eqref{as2} we have
  \begin{align*}
  &\sum_{\substack{m=(m_1,\cdots,m_r)\in(\mathbb Z^{\nu})^r\\\cc\ca\cs(m)=n}}\prod_{j=1}^{r}|m_j|^{\alpha_j}e^{-\kappa|m_j|}\\
  =~&\sum_{\substack{m=(m_1,\cdots,m_r)\in(\mathbb Z^{\nu})^r\\\cc\ca\cs(m)=n}}\prod_{j=1}^{r}|m_j|^{\alpha_j}e^{-\frac{\kappa}{2}|m_j|}\cdot \prod_{j=1}^{r}e^{-\frac{\kappa}{2}|m_j|}\\
  =~&\sum_{\substack{m=(m_1,\cdots,m_r)\in(\mathbb Z^{\nu})^r\\\cc\ca\cs(m)=n}}\prod_{j=1}^{r}|m_j|^{\alpha_j}e^{-\frac{\kappa}{2}|m_j|}\cdot e^{-\frac{\kappa}{2}\sum_{j=1}^{r}|(-1)^{j-1}m_j|}\\
    \leq&\sum_{\substack{m=(m_1,\cdots,m_r)\in(\mathbb Z^{\nu})^r\\\cc\ca\cs(m)=n}}\prod_{j=1}^{r}|m_j|^{\alpha_j}e^{-\frac{\kappa}{2}|m_j|}\cdot e^{-\frac{\kappa}{2}|\sum_{j=1}^{r}(-1)^{j-1}m_j|}\\
  =~&\sum_{\substack{m=(m_1,\cdots,m_r)\in(\mathbb Z^{\nu})^r\\\cc\ca\cs(m)=n}}\prod_{j=1}^{r}|m_j|^{\alpha_j}e^{-\frac{\kappa}{2}|m_j|}\cdot e^{-\frac{\kappa}{2}|\cc\ca\cs(m)=n|}\\
  \leq~&\sum_{\substack{m=(m_1,\cdots,m_r)\in(\mathbb Z^{\nu})^r}}\prod_{j=1}^r|m_j|^{\alpha_j}e^{-\frac{\kappa}{2}|m_j|}\cdot e^{-\frac{\kappa}{2}|n|}\\
  \leq~&e^{-\frac{\kappa}{2}|n|}(12\kappa^{-1})^{|\alpha|+\nu r}\prod_{j=1}^r\alpha_j!.
  \end{align*}
This completes the proof of Lemma \ref{4}.
\end{proof}

It follows from Lemma \ref{less}, \eqref{dimsi}, Lemma \ref{4}, \eqref{sil} and Lemma \ref{lemalpha} that
\begin{lemm}\label{lemmck}
For all $k\geq1$, we have
\begin{align}\label{dsf}
\cc_k(t,n)\ll e^{-\frac{\kappa}{2}|n|}\cdot B^{1/2}(12\kappa^{-1})^\nu\sum_{\gamma^{(k)}\in\Gamma^{(k)}}
\frac{\tT^{\ell(\gamma^{(k)})}}{\tD(\gamma^{(k)})}
\sum_{\substack{\alpha=(\alpha_j)_{1\leq j\leq2\sigma(\gamma^{(k)})}\\\in\mathfrak R^{(k,\gamma^{(k)})}}}\prod_{j=1}^{2\sigma(\gamma^{(k)})}\alpha_j!,
\end{align}
where $\tT=(12\kappa^{-1})^{2\nu+1}B|\omega|t$.
\end{lemm}
{
\begin{proof}
Set $n^{(k)}=(m_j)_{1\leq j\leq2\sigma(\gamma^{(k)})}$, where $m_j\in\bZ^\nu$ for all $j=1,\cdots,2\sigma(\gamma^{(k)})$. Then
\begin{align*}
|n^{(k)}|&=\sum_{j=1}^{2\sigma(\gamma^{(k)})}|m_j|\\
&=\sum_{j=1}^{2\sigma(\gamma^{(k)})}|(-1)^{j-1}m_j|\\
&\geq|\sum_{j=1}^{2\sigma(\gamma^{(k)})}(-1)^{j-1}m_j|\\
&=|\cc\ca\cs(n^{(k)})|.
\end{align*}
Dividing $e^{-\kappa|n^{(k)}|}$ in \eqref{cksl} into two equal parts $e^{-\kappa/2|n^{(k)}|}$,  with the above inequality, Lemma \ref{less}, Lemma \ref{3}, Lemma \ref{lemalpha}, we have
\begin{align*}
&\cc_k(t,n)\\
\ll& e^{-\frac{\kappa}{2}|n|}\cdot B^{1/2}\sum_{\gamma^{(k)}\in\Gamma^{(k)}}\frac{(B|\omega|t)^{\ell(\gamma^{(k)})}}{\tD(\gamma^{(k)})}\sum_{\substack{\alpha=(\alpha_j)_{1\leq j\leq2\sigma(\gamma^{(k)})}\\\in\mathfrak R^{(k,\gamma^{(k)})}}}(12\kappa^{-1})^{|\alpha|+2\sigma(\gamma^{(k)})\nu}\prod_{j}\alpha_j!
\\
\ll& e^{-\frac{\kappa}{2}|n|}\cdot B^{1/2}(12\kappa^{-1})^\nu\sum_{\gamma^{(k)}\in\Gamma^{(k)}}\frac{(B|\omega|(12\kappa^{-1})^{2\nu+1}t)^{\ell(\gamma^{(k)})}}{\tD(\gamma^{(k)})}
\sum_{\substack{\alpha=(\alpha_j)_{1\leq j\leq2\sigma(\gamma^{(k)})}\\\in\mathfrak R^{(k,\gamma^{(k)})}}}\prod_{j}\alpha_j!.
\end{align*}
This finishes the proof of Lemma \ref{lemmck}.
\end{proof}
}

For the estimate of the right-hand side of \eqref{dsf}, except for the constant and the exponential factor, we first have the following result:

\begin{lemm}\label{lemdd}
If $0<\tT\leq\frac{4}{81}$, then
\begin{align}\label{ttt}
M_k\triangleq\sum_{\gamma^{(k)}\in\Gamma^{(k)}}
\frac{\tT^{\ell(\gamma^{(k)})}}{\tD(\gamma^{(k)})}
\sum_{\alpha=(\alpha_j)_{1\leq j\leq2\sigma(\gamma^{(k)})}\in\mathfrak R^{(k,\gamma^{(k)})}}\prod_{j=1}^{2\sigma(\gamma^{(k)})}\alpha_j!\leq\frac{3}{2},\quad\forall k\geq1.
\end{align}
\end{lemm}
\begin{proof}
Here we give a more understandable and accessible way to obtain the desired result; compare \cite[Lemma 2.16]{MR4751185}.

For the sake of convenience, set
\[
P_k(\gamma^{(k)})=\sum_{\alpha=(\alpha_j)_{1\leq j\leq2\sigma(\gamma^{(k)})}\in\mathfrak R^{(k,\gamma^{(k)})}}\prod_{j=1}^{2\sigma(\gamma^{(k)})}\alpha_j!
\]
and
\[
M_k(\gamma^{(k)})=\frac{\tT^{\ell(\gamma^{(k)})}}{\tD(\gamma^{(k)})}
P_k(\gamma^{(k)}).
\]

It is clear that $M_k(0)\leq1$ for all $k\geq1$, and $M_1(1)\leq3\tT$. Hence $M_1\leq1+3\tT\leq\frac{3}{2}$ if $0<\tT\leq\frac{4}{81}$. This shows that \eqref{ttt} holds for $k=1$.

Let $k\geq2$. Assume that \eqref{ttt} is true for all $1<k^\prime<k$.

For $k$, set $\gamma^{(k)}=(\gamma_j^{(k-1)})_{1\leq j\leq3}\in(\Gamma^{(k-1)})^3$. It follows from the definitions of $\ell$ and $\tD$ that
\begin{align}\label{mkk}
M_k&=M_k(0)+M_k(\gamma^{(k)})\nonumber\\
&=1+\frac{\tT}{\ell(\gamma^{(k)})}\prod_{j=1}^3\frac{\tT^{\ell(\gamma_j^{(k-1)})}}{\tD(\gamma_j^{(k-1)})}
P_k(\gamma^{(k)}).
\end{align}
In this case, set $\alpha=(\alpha_j)_{1\leq j\leq2\sigma(\gamma^{(k)})}=(\alpha^{(1)},\alpha^{(2)},\alpha^{(3)})+\beta$, where
\begin{align*}
(\xi_1,\cdots,\xi_{2\sigma(\gamma_1^{(k-1)})})&=\alpha^{(1)}\in\cR^{(k-1,\gamma_1^{(k-1)})};\\
(\eta_1,\cdots,\eta_{2\sigma(\gamma_2^{(k-1)})})&=\alpha^{(2)}\in\cR^{(k-1,\gamma_2^{(k-1)})};\\
(\tau_1,\cdots,\tau_{2\sigma(\gamma_3^{(k-1)})})&=\alpha^{(3)}\in\cR^{(k-1,\gamma_3^{(k-1)})},
\end{align*}
and
\begin{align*}
\ce^{(k,\gamma^{(k)})}\ni\beta&=(\beta_j)_{1\leq j\leq2\sigma(\gamma^{(k)})}\\
&=(\xi^0_1,\cdots,\xi^0_{2\sigma(\gamma_1^{(k-1)})};\eta^0_1,\cdots,\eta^0_{2\sigma_2^{(k-1)}};
\tau^0_1,\cdots,\tau^0_{2\sigma(\gamma_3^{(k-1)})}).
\end{align*}
Hence
\begin{align*}
\alpha=(&\xi_1+\xi^0_1,\cdots,\xi_{2\sigma(\gamma_1^{(k-1)})}+\xi^0_{2\sigma(\gamma_1^{(k-1)})};\\
&\eta_1+\eta^0_1,\cdots,
\eta_{2\sigma(\gamma_2^{(k-1)})}+\eta^0_{2\sigma(\gamma_2^{(k-1)})};\\
&\tau_1+\tau^0_1,\cdots,\tau_{2\sigma(\gamma_3^{(k-1)})}+\tau^0_{2\sigma(\gamma_3^{(k-1)})}).
\end{align*}
Since $|\beta|=1$, there is exactly one component that is $1$ and the rest of the components are $0$. So we can split $P_k(\gamma^{(k)})$ into sums over $j_0\in\{1,\cdots,2\sigma(\gamma^{(k)})\}$ of
$$\{\beta\in\ce^{(k,\gamma^{(k)})}: \text{the}~j_0\text{-th component}~\beta_{j_0} \text{of}~\beta~\text{is}~1\}$$ as follows (we will use the symbol $e_{j_0}=(0,\cdots,0,1,0,\cdots,0)$, that is, the $j_0$-th component is $1$ and the rest is zero)
\begin{align*}
P_k(\gamma^{(k)})=&\sum_{j_0=1}^{2\sigma(\gamma^{(k)})}\sum_{\beta_{j_0}=1}
\sum_{\substack{\alpha=(\alpha_j)_{1\leq j\leq2\sigma(\gamma^{(k)})}=(\alpha^{(1)},\alpha^{(2)},\alpha^{(3)})+\beta_{j_0}e_{j_0}\\(\alpha^{(1)},\alpha^{(2)},\alpha^{(3)})
\in\prod_{j=1}^3\cR^{(k-1,
\gamma_j^{(k-1)})}}}\alpha_j!\\
\triangleq& P^{(1)}_k(\gamma^{(k)})+P^{(2)}_k(\gamma^{(k)})+P^{(3)}_k(\gamma^{(k)}),
\end{align*}
where
\begin{align*}
P^{(1)}_k(\gamma^{(k)})&=\sum_{j_1=1}^{2\sigma(\gamma_1^{(k-1)})}\sum_{\substack{\alpha^{(1)}=(\xi_j)_{1\leq j\leq2\sigma(\gamma_{1}^{(k-1)})}\\\in\cR^{(k-1,\gamma_{1}^{(k-1)})}\\\alpha^{(2)}=(\eta_j)_{1\leq j\leq2\sigma(\gamma_{2}^{(k-1)})}\\\in\cR^{(k-1,\gamma_{2}^{(k-1)})}\\\alpha^{(3)}=(\tau_j)_{1\leq j\leq2\sigma(\gamma_{3}^{(k-1)})}\\\in\cR^{(k-1,\gamma_{3}^{(k-1)})}}}\prod_{j=1,j\neq j_1}^{2\sigma(\gamma_1^{(k-1)})}\xi_j!\cdot(\xi_{j_1}+1)!\prod_{j=1}^{2\sigma(\gamma_2^{(k-1)})}
\eta_j!\prod_{j=1}^{2\sigma(\gamma_1^{(k-1)})}\tau_j!;\\
P^{(2)}_k(\gamma^{(k)})&=\sum_{j_2=1}^{2\sigma(\gamma_2^{(k-1)})}\sum_{\substack{\alpha^{(1)}=(\xi_j)_{1\leq j\leq2\sigma(\gamma_{1}^{(k-1)})}\\\in\cR^{(k-1,\gamma_{1}^{(k-1)})}\\\alpha^{(2)}=(\eta_j)_{1\leq j\leq2\sigma(\gamma_{2}^{(k-1)})}\\\in\cR^{(k-1,\gamma_{2}^{(k-1)})}\\\alpha^{(3)}=(\tau_j)_{1\leq j\leq2\sigma(\gamma_{3}^{(k-1)})}\\\in\cR^{(k-1,\gamma_{3}^{(k-1)})}}}\prod_{j=1}^{2\sigma(\gamma_1^{(k-1)})}\xi_j!
\prod_{j=1,j\neq j_2}^{2\sigma(\gamma_2^{(k-1)})}
\eta_j!\cdot(\eta_{j_2}+1)!\prod_{j=1}^{2\sigma(\gamma_1^{(k-1)})}\tau_j;\\
P^{(3)}_k(\gamma^{(k)})&=\sum_{j_3=1}^{2\sigma(\gamma_3^{(k-1)})}\sum_{\substack{\alpha^{(1)}=(\xi_j)_{1\leq j\leq2\sigma(\gamma_{1}^{(k-1)})}\\\in\cR^{(k-1,\gamma_{1}^{(k-1)})}\\\alpha^{(2)}=(\eta_j)_{1\leq j\leq2\sigma(\gamma_{2}^{(k-1)})}\\\in\cR^{(k-1,\gamma_{2}^{(k-1)})}\\\alpha^{(3)}=(\tau_j)_{1\leq j\leq2\sigma(\gamma_{3}^{(k-1)})}\\\in\cR^{(k-1,\gamma_{3}^{(k-1)})}}}\prod_{j=1}^{2\sigma(\gamma_1^{(k-1)})}\xi_j!
\prod_{j=1}^{2\sigma(\gamma_2^{(k-1)})}
\eta_j!\prod_{j=1,j\neq j_3}^{2\sigma(\gamma_1^{(k-1)})}\tau_j!
\cdot(\tau_{j_3}+1)!.
\end{align*}
Take $P^{(1)}_k(\gamma_1^{(k-1)})$ as an example to make a further analysis as follows.
\begin{align*}
P^{(1)}_k(\gamma^{(k)})=&\sum_{j_1=1}^{2\sigma(\gamma_1^{(k-1)})}\sum_{\substack{\alpha^{(1)}=(\xi_j)_{1\leq j\leq2\sigma(\gamma_{1}^{(k-1)})}\\\in\cR^{(k-1,\gamma_{1}^{(k-1)})}\\\alpha^{(2)}=(\eta_j)_{1\leq j\leq2\sigma(\gamma_{2}^{(k-1)})}\\\in\cR^{(k-1,\gamma_{2}^{(k-1)})}\\\alpha^{(3)}=(\tau_j)_{1\leq j\leq2\sigma(\gamma_{3}^{(k-1)})}\\\in\cR^{(k-1,\gamma_{3}^{(k-1)})}}}(\xi_{j_1}+1)\prod_{j=1}^{2\sigma(\gamma_1^{(k-1)})}\xi_j!\prod_{j=1}^{2\sigma(\gamma_2^{(k-1)})}
\eta_j!\prod_{j=1}^{2\sigma(\gamma_1^{(k-1)})}\tau_j!\\
=&\sum_{\substack{\alpha^{(1)}=(\xi_j)_{1\leq j\leq2\sigma(\gamma_{1}^{(k-1)})}\\\in\cR^{(k-1,\gamma_{1}^{(k-1)})}\\\alpha^{(2)}=(\eta_j)_{1\leq j\leq2\sigma(\gamma_{2}^{(k-1)})}\\\in\cR^{(k-1,\gamma_{2}^{(k-1)})}\\\alpha^{(3)}=(\tau_j)_{1\leq j\leq2\sigma(\gamma_{3}^{(k-1)})}\\\in\cR^{(k-1,\gamma_{3}^{(k-1)})}}}\sum_{j_1=1}^{2\sigma(\gamma_1^{(k-1)})}(\xi_{j_1}+1)\prod_{j=1}^{2\sigma(\gamma_1^{(k-1)})}\xi_j!\prod_{j=1}^{2\sigma(\gamma_2^{(k-1)})}
\eta_j!\prod_{j=1}^{2\sigma(\gamma_1^{(k-1)})}\tau_j!\\
=&\sum_{\substack{\alpha^{(1)}=(\xi_j)_{1\leq j\leq2\sigma(\gamma_{1}^{(k-1)})}\\\in\cR^{(k-1,\gamma_{1}^{(k-1)})}\\\alpha^{(2)}=(\eta_j)_{1\leq j\leq2\sigma(\gamma_{2}^{(k-1)})}\\\in\cR^{(k-1,\gamma_{2}^{(k-1)})}\\\alpha^{(3)}=(\tau_j)_{1\leq j\leq2\sigma(\gamma_{3}^{(k-1)})}\\\in\cR^{(k-1,\gamma_{3}^{(k-1)})}}}\left(\ell(\gamma_1^{(k-1)}+2\sigma(\gamma_1^{(k-1)}))\right)\prod_{j=1}^{2\sigma(\gamma_1^{(k-1)})}\xi_j!\prod_{j=1}^{2\sigma(\gamma_2^{(k-1)})}
\eta_j!\prod_{j=1}^{2\sigma(\gamma_1^{(k-1)})}\tau_j!\\
=&(3\ell(\gamma_1^{(k-1)})+1)\sum_{\substack{\alpha^{(1)}=(\xi_j)_{1\leq j\leq2\sigma(\gamma_{1}^{(k-1)})}\\\in\cR^{(k-1,\gamma_{1}^{(k-1)})}\\\alpha^{(2)}=(\eta_j)_{1\leq j\leq2\sigma(\gamma_{2}^{(k-1)})}\\\in\cR^{(k-1,\gamma_{2}^{(k-1)})}\\\alpha^{(3)}=(\tau_j)_{1\leq j\leq2\sigma(\gamma_{3}^{(k-1)})}\\\in\cR^{(k-1,\gamma_{3}^{(k-1)})}}}\prod_{j=1}^{2\sigma(\gamma_1^{(k-1)})}\xi_j!\prod_{j=1}^{2\sigma(\gamma_2^{(k-1)})}
\eta_j!\prod_{j=1}^{2\sigma(\gamma_1^{(k-1)})}\tau_j!\\
=&(3\ell(\gamma_1^{(k-1)})+1)\sum_{\substack{\alpha^{(1)}=(\xi_j)_{1\leq j\leq2\sigma(\gamma_{1}^{(k-1)})}\in\cR^{(k-1,\gamma_{1}^{(k-1)})}}}\prod_{j=1}^{2\sigma(\gamma_1^{(k-1)})}\xi_j!\\
&\sum_{\substack{\alpha^{(2)}=(\eta_j)_{1\leq j\leq2\sigma(\gamma_{2}^{(k-1)})}\in\cR^{(k-1,\gamma_{2}^{(k-1)})}}}\prod_{j=1}^{2\sigma(\gamma_2^{(k-1)})}
\eta_j!\\
&\sum_{\substack{\alpha^{(3)}=(\tau_j)_{1\leq j\leq2\sigma(\gamma_{3}^{(k-1)})}\in\cR^{(k-1,\gamma_{3}^{(k-1)})}}}\prod_{j=1}^{2\sigma(\gamma_3^{(k-1)})}\tau_j!\\
=&(3\ell(\gamma_1^{(k-1)})+1)\prod_{j=1}^3P_{k-1}(\gamma_j^{(k-1)}).
\end{align*}
Similarly,
\begin{align*}
P^{(2)}_k(\gamma^{(k)})=&(3\ell(\gamma_2^{(k-1)})+1)\prod_{j=1}^3P_{k-1}(\gamma_j^{(k-1)});\\
P^{(3)}_k(\gamma^{(k)})=&(3\ell(\gamma_3^{(k-1)})+1)\prod_{j=1}^3P_{k-1}(\gamma_j^{(k-1)}).
\end{align*}
Thus
\begin{align}
P_k(\gamma^{(k)})=&\sum_{j=1}^3P^{(j)}_{k}(\gamma^{(k)})\nonumber\\
=&\sum_{j=1}^3(3\ell(\gamma_j^{(k-1)})+1)\prod_{j=1}^3P_{k-1}(\gamma_j^{(k-1)})\nonumber\\
=&3\left(1+\sum_{j=1}^3\ell(\gamma_j^{(k-1)})\right)\prod_{j=1}^3P_{k-1}(\gamma_j^{(k-1)})\nonumber\\
=&3\ell(\gamma^{(k)})\prod_{j=1}^3P_{k-1}(\gamma_j^{(k-1)}).\label{ppp}
\end{align}
Plugging \eqref{ppp} into \eqref{mkk} yields that
\begin{align*}
M_k=&M_k(0)+M_k(\gamma^{(k)})\\
=&1+\frac{\tT}{\ell(\gamma^{(k)})}\cdot3\ell(\gamma^{(k)})\prod_{j=1}^3P_{k-1}(\gamma_j^{(k-1)})\\
\leq&1+3\tT\cdot(3/2)^3\\
\leq&1+3\cdot\frac{4}{81}\cdot\frac{27}{8}\\
=&\frac{3}{2}.
\end{align*}
This proves that \eqref{ttt} holds for $k$, and hence it holds for all $k\geq1$ by induction. This completes the proof of Lemma \ref{lemdd}.
\end{proof}

By Lemma \ref{lemdd}, we know that
\begin{align}\label{sskk}
\sum_{\gamma^{(k)}\in\Gamma^{(k)}}
\frac{\tT^{\ell(\gamma^{(k)})}}{\tD(\gamma^{(k)})}
\sum_{\alpha=(\alpha_j)_{1\leq j\leq2\sigma(\gamma^{(k)})}\in\mathfrak R^{(k,\gamma^{(k)})}}\prod_{j=1}^{2\sigma(\gamma^{(k)})}\alpha_j!\leq\frac{3}{2}
\end{align}
provided that
\begin{align}\label{t2}
0<t\leq\frac{4\cdot\kappa^{2\nu+1}}{81\cdot12^{2\nu+1}B|\omega|}\triangleq t_2.
\end{align}

Plugging \eqref{sskk} into \eqref{dsf}, we can prove that $\cc_k(t,n)$ is $\kappa/2$-exponentially decaying uniformly in $(t,n)\in[0,t_2]\times\bZ^\nu$.
\begin{lemm}
If $0<t\leq t_2$, then
\begin{align}\label{ckc}
\cc_k(t,n)\ll Ce^{-\frac{\kappa}{2}|n|},\quad\forall n\in\bZ^\nu,~\forall k\geq1,
\end{align}
where $C=\frac{3}{2}B^{1/2}(12\kappa^{-1})^\nu$.
\end{lemm}
{
\begin{proof}
It follows from Lemmas \ref{lemmck}-\ref{lemdd} that \eqref{ckc} holds true for all $k\geq1$.
\end{proof}
}

\subsubsection{The Picard Sequence is a Cauchy Sequence}

Here we use \eqref{gsj}, \eqref{ckc} and Lemma \ref{lemmck}
and induction to prove that the Picard sequence is a Cauchy sequence; see Lemma \ref{lemmc}.

\begin{lemm}\label{lemmc}
For all $k\geq1$, we have
\begin{align}
&\cc_{k}(t,n)-\cc_{k-1}(t,n)\nonumber\\
\ll&\frac{3^{k-1}C^{2k+1}(|\omega|t)^k}{k!}\sum_{\substack{m^{(k)}=(m_j)_{1\leq j\leq2k+1}\in(\bZ^\nu)^{2k+1}\\\cc\ca\cs(m^{(k)})=n}}\sum_{\substack{\alpha=(\alpha_j)_{1\leq j\leq2k+1}\in\tG^{(k)}}}\prod_{j=1}^{2k+1}|m_j|^{\alpha_j}e^{-\frac{\kappa}{2}|m_j|}\label{1c}\\
\ll&e^{-\frac{\kappa}{4}|n|}C^\prime\left(12eC^2(24\kappa^{-1})^{2\nu+1}|\omega|t\right)^k,\label{2cc}
\end{align}
where $C^\prime=3^{-1}C(24\kappa^{-1})^{\nu}e^{\frac{1}{2}}C$, and
\begin{align*}
\tG^{(k)}&:=
\begin{cases}
\{(1,0,0),(0,1,0),(0,0,1)\}, &k=1;\\
\ce^{(k)}+\tG^{(k-1)}\times\{0\in\bZ\}\times\{0\in\bZ\}, &k\geq2;
\end{cases}\\
\ce^{(k)}&:=\{\alpha\in\bN^{2k+1}: |\alpha|=1\},\quad k\geq2.
\end{align*}

Furthermore, for all $q\in\bN$, we have
\begin{align}\label{cauchy}
\cc_{k+q}(t,n)-\cc_{k}(t,n)\ll e^{-\frac{\kappa}{4}|n|}C^{\prime\prime}\left(12eC^2(24\kappa^{-1})^{2\nu+1}|\omega|t\right)^{k},
\end{align}
where $C^{\prime\prime}:=\frac{C^\prime}{1-12eC^2(24\kappa^{-1})^{2\nu+1}|\omega|t}$.
This shows that the Picard sequence $\{\cc_k(t,n)\}$ is a Cauchy sequence.
\end{lemm}
\begin{proof}
For $k=1$, by the exponential decay property
\eqref{ckc}, we have
\begin{align*}
\cc_1(t,n)-\cc_0(t,n)&\ll|\omega||n|\int_0^t\sum_{\substack{n_1,n_2,n_3\in\bZ^\nu\\n_1-n_2+n_3=n}}\prod_{j=1}^3|\{\cc_0(s,n_j)\}^{\ast^{[j-1]}}|{\rm d}s\\
&\ll|\omega||n|\int_0^t\sum_{\substack{m^{(1)}=(m_1,m_2,m_3)\in(\bZ^\nu)^3\\\cc\ca\cs(m^{(1)})=n}}\prod_{j=1}^3C
e^{-\frac{\kappa}{2}|m_j|}{\rm d}s\\
&\ll C^3|\omega|t\sum_{\substack{m^{(1)}=(m_1,m_2,m_3)\in(\bZ^\nu)^3\\\cc\ca\cs(m^{(1)})=n}}\sum_{\substack{\alpha=(\alpha_j)_{1\leq j\leq3}\in\tG^{(1)}}}\prod_{j=1}^3|m_j|^{\alpha_j}e^{-\frac{\kappa}{2}|m_j|}.
\end{align*}
This shows that \eqref{1c} holds for $k=1$.

Let $k\geq2$. Assume that \eqref{1c} is true for all $1<k^\prime<k$.

For $k$, it follows from \eqref{gsj} and the following decomposition
\begin{align}\label{deco}
|\prod_{j=1}^{j_0} a_j-\prod_{j=1}^{j_0} b_j|
\leq\sum_{J=1}^{j_0}\prod_{j=1}^{J-1}|b_j|\cdot|a_{J}-
b_{J}|\cdot\prod_{j=J+1}^{{j_0}}|a_j|,
\end{align}
where
\begin{align}\label{deco2}
\prod_{j=1}^0|b_j|:=1\quad\text{and}\quad\prod_{j={j_0}+1}^{j_0}|a_j|:=1
\end{align}
that
\begin{align}\label{kk1}
\cc_k(t,n)-\cc_{k-1}(t,n)\ll&|\omega||n|\int_0^t\sum_{\substack{n_1,n_2,n_3\in\bZ^\nu\\n_1-n_2+n_3=n}}|
\prod_{j=1}^3\{\cc_1(s,n_j)\}^{\ast^{[j-1]}}-\prod_{j=1}^3\{\cc_0(s,n_j)\}^{\ast^{[j-1]}}|{\rm d}s\nonumber\\
\ll&\Phi_1+\Phi_2+\Phi_3,
\end{align}
where
\begin{align*}
\Phi_1&=|\omega||n|\int_0^t\sum_{\substack{n_1,n_2,n_3\in\bZ^\nu\\n_1-n_2+n_3=n}}
|\cc_1(s,n_1)-\cc_0(s,n_1)||\cc_1(s,n_2)||\cc_1(s,n_3)|{\rm d}s;\\
\Phi_2&=|\omega||n|\int_0^t\sum_{\substack{n_1,n_2,n_3\in\bZ^\nu\\n_1-n_2+n_3=n}}
|\cc_0(s,n_1)||\cc_1(s,n_2)-\cc_0(s,n_2)||\cc_1(s,n_3)|{\rm d}s;\\
\Phi_3&=|\omega||n|\int_0^t\sum_{\substack{n_1,n_2,n_3\in\bZ^\nu\\n_1-n_2+n_3=n}}
|\cc_0(s,n_1)||\cc_0(s,n_2)||\cc_1(s,n_3)-\cc_0(s,n_3)|{\rm d}s.
\end{align*}

For $\Phi_1$, it follows from the induction hypothesis and the exponential decay property
\eqref{ckc} that
\begin{align}\label{epew}
\Phi_1\ll&|\omega||n|\int_0^t\sum_{\substack{n_1,n_2,n_3\in\bZ^\nu\\n_1-n_2+n_3=n}}
\frac{3^{k-2}C^{2k-1}(|\omega|s)^{k-1}}{(k-1)!}{\rm d}s\nonumber\\
&\sum_{\substack{m^{(k-1)}=(m_j)_{1\leq j\leq2k-1}\in(\bZ^\nu)^{2k-1}\\\cc\ca\cs(m^{(k-1)})=n_1}}\sum_{\substack{\alpha=(\alpha_j)_{1\leq j\leq2k-1}\in\tG^{(k-1)}}}\prod_{j=1}^{2k-1}|m_j|^{\alpha_j}e^{-\frac{\kappa}{2}|m_j|}Ce^{-\frac{\kappa}{2}|n_2|}
Ce^{-\frac{\kappa}{2}|n_3|}
\nonumber\\
=&\frac{3^{k-2}C^{2k+1}(|\omega|t)^k}{k!}\sum_{\substack{m^{(k)}=(m_j)_{1\leq j\leq 2k+1}\in(\bZ^\nu)^{2k+1}\\\cc\ca\cs(m^{(k)})=n}}
\sum_{\alpha\in\ce^{(k)}}\prod_{j=1}^{2k+1}|m_j|^{\alpha_j}\nonumber\\
&\sum_{\alpha=(\alpha_j)_{1\leq j\leq2k+1}\in\tG^{(k-1)}\times\{0\}\times\{0\}}\prod_{j=1}^{2k+1}|m_j|^{\alpha_j}e^{-\frac{\kappa}{2}|m_j|}
\nonumber\\
=&\frac{3^{k-2}C^{2k+1}(|\omega|t)^k}{k!}\sum_{\substack{m^{(k)}=(m_j)_{1\leq j\leq 2k+1}\in(\bZ^\nu)^{2k+1}\\\cc\ca\cs(m^{(k)})=n}}\sum_{\alpha\in\tG^{(k)}}\prod_{j=1}^{2k+1}|m_j|^{\alpha_j}
e^{-\frac{\kappa}{2}|m_j|}
.\end{align}
Here we use $m_{2k}$ and $m_{2k+1}$ to label $n_2$ and $n_3$ respectively, and set
$$
m^{(k)}=(m^{(k-1)},m_{2k},m_{2k+1})\in(\bZ^\nu)^{2k-1}\times\bZ^\nu\times\bZ^\nu\simeq(\bZ^\nu)^{2k+1}.
$$
Hence
\begin{align*}
n&=n_1-n_2+n_3\\
&=\cc\ca\cs(m^{(k-1)})-m_{2k}+m_{2k+1}\\
&=\cc\ca\cs(m^{(k)}).
\end{align*}

Analogously, for $j=2,3$, we have
\begin{align}\label{epews}
\Phi_j\ll\frac{3^{k-2}C^{2k+1}(|\omega|t)^k}{k!}\sum_{\substack{m^{(k)}=(m_j)_{1\leq j\leq 2k+1}\in(\bZ^\nu)^{2k+1}\\\cc\ca\cs(n^{(k)})=n}}\sum_{\alpha\in\tG^{(k)}}\prod_{j=1}^{2k+1}|m_j|^{\alpha_j}e^{-\frac{\kappa}{2}|m_j|}
.\end{align}

Plugging \eqref{epew}-\eqref{epews} into \eqref{kk1} yields that
\[
\cc_k(t,n)-\cc_{k-1}(t,n)\ll\frac{3^{k-1}C^{2k+1}(|\omega|t)^k}{k!}\sum_{\substack{m^{(k)}=(m_j)_{1\leq j\leq 2k+1}\in(\bZ^\nu)^{2k+1}\\\cc\ca\cs(n^{(k)})=n}}\sum_{\alpha\in\tG^{(k)}}\prod_{j=1}^{2k+1}|m_j|^{\alpha_j}e^{-\frac{\kappa}{2}|m_j|}
.\]
This proves that \eqref{1c} holds for $k$, and hence for all $k\geq1$ by induction.

Next we prove \eqref{2cc}. It follows from Lemmas \ref{4} and \ref{lemg} that
\begin{align}
\sum_{\substack{m^{(k)}=(m_j)_{1\leq j\leq 2k+1}\in(\bZ^\nu)^{2k+1}\\\cc\ca\cs(n^{(k)})=n}}\prod_{j=1}^{2k+1}|m_j|^{\alpha_j}e^{-\frac{\kappa}{2}|m_j|}
&\ll e^{-\frac{\kappa}{4}|n|}(24\kappa^{-1})^{|\alpha|+(2k+1)\nu}\prod_{j=1}^{2k+1}\alpha_j!\nonumber\\
&=e^{-\frac{\kappa}{4}|n|}(24\kappa^{-1})^{\nu}((24\kappa^{-1})^{2\nu+1})^k\prod_{j=1}^{2k+1}\alpha_j!.\label{alsx}
\end{align}
Inserting \eqref{alsx} into \eqref{1c}, we have
\begin{align}\label{deno}
\cc_k(t,n)-\cc_{k-1}(t,n)\ll e^{-\frac{\kappa}{4}|n|}(24\kappa^{-1})^{\nu}\frac{3^{k-1}C^{2k+1}((24\kappa^{-1})^{2\nu+1}|\omega|t)^k}{k!}
\sum_{\alpha\in\tG^{(k)}}\prod_{j=1}^{2k+1}\alpha_j!.
\end{align}
By Lemma \ref{lemg}, we know that $k=L<N=2k+1$.  It follows from Lemma \ref{aaaa} that
\begin{align}\label{gl}
\sum_{\alpha\in\tG^{(k)}}\prod_{j=1}^{2k+1}\alpha_j!\ll(2(2k+1))^k.
\end{align}
It follows from Stirling's formula (see Lemma \ref{stf}) that the growth of $k^k$ can be balanced by $k!$ (up to an exponential factor) appearing in the denominator of the right-hand side of \eqref{deno}, or rather,
\begin{align}
\frac{(2(2k+1))^{k}}{k!}&\ll\frac{(2(2k+1))^{k}}{k^ke^{-k}}\nonumber\\
&=4^ke^k\left(1+\frac{1}{2k}\right)^k\nonumber\\
&\ll4^ke^{k+\frac{1}{2}}.\label{lsf}
\end{align}
Plugging  \eqref{gl} and \eqref{lsf} into \eqref{deno} yields that
\[\cc_k(t,n)-\cc_{k-1}(t,n)\ll e^{-\frac{\kappa}{4}|n|}C^\prime\left(12eC^2(24\kappa^{-1})^{2\nu+1}|\omega|t\right)^k.\]
This proves \eqref{2cc}.

Lastly, we prove \eqref{cauchy}. It follows from the triangle inequality that
\begin{align*}
\cc_{k+q}(t,n)-\cc_k(t,n)&=\sum_{j=1}^q\cc_{k+j}(t,n)-\cc_{k+j-1}(t,n)\\
&\ll e^{-\frac{\kappa}{4}|n|}C^\prime\sum_{j=1}^q\left(12eC^2(24\kappa^{-1})^{2\nu+1}|\omega|t\right)^{k+j}\\
&\ll e^{-\frac{\kappa}{4}|n|}C^{\prime\prime}\left(12eC^2(24\kappa^{-1})^{2\nu+1}|\omega|t\right)^{k},
\end{align*}
provided that
\begin{align}\label{t3}
0<t<\frac{1}{12eC^2(24\kappa^{-1})^{2\nu+1}|\omega|}:=t_3.
\end{align}
This shows that \eqref{cauchy} holds true and completes the proof of Lemma \ref{lemmc}.
\end{proof}

\subsubsection{Existence and Convergence}

From the argument in the last subsubsection we know that the Picard sequence $\{\cc_k(t,n)\}$ is a Cauchy sequence. Let $\cc(t,n)$ be the limit function defined on $[0,t_1)\times\bZ^\nu$, where $t_1:=\min\{t_2,t_3\}$.
According to the triangle inequality, we have
\begin{align}\label{skd}
\cc(t,n)\ll Ce^{-\frac{\kappa}{2}|n|},\quad\forall(t,n)\in[0,t_1)\times\bZ^\nu.
\end{align}

Set
\begin{align*}
u(t,x)&=\sum_{n\in\bZ^\nu}\cc(t,n)e^{{\rm i}\langle n\rangle x};\\
(\partial^j_xu)(t,x)&=\sum_{n\in\bZ^\nu}({\rm i}\langle n\rangle)^j\cc(t,n)e^{{\rm i}\langle n\rangle x},\quad j=1,2;\\
(\partial_tu)(t,x)&=\sum_{n\in\bZ^\nu}\left\{-{\rm i}\langle n\rangle^2\cc(t,n)+{\rm i}\langle n\rangle\sum_{\substack{n_1,n_2,n_3\in\bZ^\nu\\n_1-n_2+n_3=n}}\prod_{j=1}^{3}\{\cc(t,n_j)\}^{\ast^{[j-1]}}e^{{\rm i}\langle n\rangle x}\right\}e^{{\rm i}\langle n\rangle x}.
\end{align*}

It follows from \eqref{skd} and Lemma \ref{l1l} that
\begin{align*}
\sum_{n\in\bZ^\nu}|\cc(t,n)|&\ll C\sum_{n\in\bZ^\nu}e^{-\frac{\kappa}{2}|n|}\\
&\ll C(6\kappa^{-1})^\nu;\\[1mm]
\sum_{n\in\bZ^\nu}|n|^2|\cc(t,n)|&\ll C\sum_{n\in\bZ^\nu}|n|^2e^{-\frac{\kappa}{2}|n|}\\
&=C\sum_{n\in\bZ^\nu}\underbrace{|n|^2e^{-\frac{\kappa}{4}|n|}}_{\text{bounded by}~2(4\kappa^{-1})^{2}}e^{-\frac{\kappa}{4}|n|}\\
&\ll 2(4\kappa^{-1})^{2}(12\kappa^{-1})^\nu C;
\end{align*}
and
\begin{align*}
\sum_{n\in\bZ^\nu}|n|\sum_{\substack{n_1,n_2,n_3\in\bZ^\nu\\n_1-n_2+n_3=n}}\prod_{j=1}^{3}|\cc(t,n_j)|&\ll
\sum_{n\in\bZ^\nu}|n|\sum_{\substack{n_1,n_2,n_3\in\bZ^\nu\\n_1-n_2+n_3=n}}\prod_{j=1}^{3} Ce^{-\frac{\kappa}{2}|n_j|}\\
&\ll C^3\sum_{n\in\bZ^\nu}|n|e^{-\frac{\kappa}{4}|n|}\sum_{n_1,n_2,n_3\in\bZ^\nu}\prod_{j=1}^{3}e^{-\frac{\kappa}{4}|n_j|}\\
&\ll C^3(12\kappa^{-1})^{3\nu}\sum_{n\in\bZ^\nu}\underbrace{|n|e^{-\frac{\kappa}{8}|n|}}_{\text{bounded by}~8\kappa^{-1}}e^{-\frac{\kappa}{8}|n|}\\
&\ll C^3(12\kappa^{-1})^{3\nu}8\kappa^{-1}\sum_{n\in\bZ^\nu}e^{-\frac{\kappa}{8}|n|}\\
&\ll C^3(12\kappa^{-1})^{3\nu}8\kappa^{-1}(24\kappa^{-1})^\nu.
\end{align*}
Hence $u, \partial_x^ju$ and $\partial_tu$ are well defined, and $u$ is a classical spatially quasi-periodic solution to \eqref{dnls} 
with initial data \eqref{id}.

\subsubsection{Uniqueness}

In this subsubsection we prove the uniqueness part of Theorem \ref{dnlsth}.

Let $\cc$ and $\cd$ be two functions of $(t,n)\in[0,t_1)\times\bZ^\nu$. Assume that they satisfy:
\begin{itemize}
  \item (same initial data)~$\cc(0,n)=\cd(0,n)$ for all $n\in\bZ^\nu$;
  \item (exponential decay property)~there exists $(C_1,\rho)\in(0,\infty)\times(0,1]$, where $\rho=\kappa/2$, such that
\[\cc(t,n)\ll C_1e^{-\rho|n|}\quad\text{and}\quad\cd(t,n)\ll C_1e^{-\rho|n|},\quad\forall(t,n)\in[0,t_1)\times\bZ^\nu;\]
  \item (integral equations)~for all $n\in\bZ^\nu$,
  \begin{align*}
  \cc(t,n)&=e^{-{\rm i}\langle n\rangle^2t}\cc(0,n)+{\rm i}\langle n\rangle\int_0^te^{-{\rm i}\langle n\rangle^2(t-s)}\sum_{\substack{n_1,n_2,n_3\in\bZ^\nu\\n_1-n_2+n_3=n}}\prod_{j=1}^{3}\{\cc(s,n_j)\}^{\ast^{[j-1]}}{\rm d}s;\\
  \cd(t,n)&=e^{-{\rm i}\langle n\rangle^2t}\cd(0,n)+{\rm i}\langle n\rangle\int_0^te^{-{\rm i}\langle n\rangle^2(t-s)}\sum_{\substack{n_1,n_2,n_3\in\bZ^\nu\\n_1-n_2+n_3=n}}\prod_{j=1}^{3}\{\cd(s,n_j)\}^{\ast^{[j-1]}}{\rm d}s.
  \end{align*}
\end{itemize}
\begin{lemm}
For all $k\geq1$, we have
\begin{align}
&\cc(t,n)-\cd(t,n)\nonumber\\
\ll&\frac{3^{2k-1}\times2C_1^{2k+1}(|\omega|t)^k}{k!}\sum_{\substack{m^{(k)}=(m_j)_{1\leq j\leq2k+1}\in(\bZ^\nu)^{2k+1}\\\cc\ca\cs(m^{(k)})=n}}\sum_{\alpha=(\alpha_j)_{1\leq j\leq2k+1}\in\tG^{(k)}}\prod_{j=1}^{2k+1}|m_j|^{\alpha_j}e^{-\frac{\kappa}{2}|m_j|}\label{cddd}\\
\ll&2C_1e^{1/2}(24\kappa^{-1})^\nu (36C_1^2e(24\kappa^{-1})^{2\nu+1}|\omega|t)^ke^{-\frac{\rho}{2}|n|}.
\end{align}
This implies that $\cc(t,n)\equiv\cd(t,n)$ for all $(t,n)\in[0,t_4)\times\bZ^\nu$; see \eqref{t4}.

\end{lemm}
\begin{proof}
We first have
\begin{align}
\cc(t,n)-\cd(t,n)&\ll|\omega||n|\int_0^t\sum_{\substack{n_1,n_2,n_3\in\bZ^\nu\\n_1-n_2+n_3=n}}|
\prod_{j=1}^{3}\{\cc(s,n_j)\}^{\ast^{[j-1]}}-\prod_{j=1}^{3}\{\cd(s,n_j)\}^{\ast^{[j-1]}}|{\rm d}s\nonumber\\
&\ll\Psi_1+\Psi_2+\Psi_3,\label{cdd}
\end{align}
where
\begin{align*}
\Psi_1&=|\omega||n|\int_0^t\sum_{\substack{n_1,n_2,n_3\in\bZ^\nu\\n_1-n_2+n_3=n}}|
\cc(s,n_1)-\cd(s,n_1)||\cc(s,n_2)||\cc(s,n_3)|{\rm d}s;\\
\Psi_2&=|\omega||n|\int_0^t\sum_{\substack{n_1,n_2,n_3\in\bZ^\nu\\n_1-n_2+n_3=n}}|\cd(s,n_1)||
\cc(s,n_2)-\cd(s,n_2)||\cc(s,n_3)|{\rm d}s;\\
\Psi_3&=|\omega||n|\int_0^t\sum_{\substack{n_1,n_2,n_3\in\bZ^\nu\\n_1-n_2+n_3=n}}|\cd(s,n_1)||\cd(s,n_2)||
\cc(s,n_3)-\cd(s,n_3)|{\rm d}s.
\end{align*}

For $\Psi_1$, we have
\begin{align}
\Psi_1
&\ll|\omega||n|\int_0^t\sum_{\substack{n_1,n_2,n_3\in\bZ^\nu\\n_1-n_2+n_3=n}}(|
\cc(s,n_1)|+|\cd(s,n_1)|)|\cc(s,n_2)||\cc(s,n_3)|{\rm d}s\nonumber\\
&\ll 2C_1^3|\omega|t|n|\sum_{\substack{n_1,n_2,n_3\in\bZ^\nu\\n_1-n_2+n_3=n}}\prod_{j=1}^3e^{-\rho|n_j|}\nonumber\\
&\ll 2C_1^3|\omega|t\sum_{\substack{m^{(1)}=(m_1,m_2,m_3)\in(\bZ^\nu)^3\\\cc\ca\cs(m^{(1)})=n}}\sum_{\alpha\in\tG^{(1)}}\prod_{j=1}^3|m_j|^{\alpha_j}e^{-\rho|m_j|}.\label{psi1}
\end{align}
Similarly, for $j=2,3$, we have
\begin{align}\label{psij}
\Psi_j\ll2C_1^3|\omega|t\sum_{\substack{m^{(1)}=(m_1,m_2,m_3)\in(\bZ^\nu)^3\\\cc\ca\cs(m^{(1)})=n}}\sum_{\alpha\in\tG^{(1)}}\prod_{j=1}^3|m_j|^{\alpha_j}e^{-\rho|m_j|}.
\end{align}
Plugging \eqref{psi1} and \eqref{psij} into \eqref{cdd} yields that
\begin{align}\label{112}
\cc(t,n)-\cd(t,n)\ll3\times
2C_1^3|\omega|t\sum_{\substack{m^{(1)}=(m_1,m_2,m_3)\in(\bZ^\nu)^3\\\cc\ca\cs(m^{(1)})=n}}\sum_{\alpha\in\tG^{(1)}}\prod_{j=1}^3|m_j|^{\alpha_j}e^{-\rho|m_j|}
,\quad\forall n\in\bZ^\nu.\end{align}
This shows that \eqref{cddd} holds for $k=1$.

Let $k\geq2$. Assume \eqref{cddd} holds for all $1<k^\prime<k$. For $k$, by induction hypothesis, we have
\begin{align*}
\Psi_1\ll& |\omega||n|\int_0^t\sum_{\substack{n_1,n_2,n_3\in\bZ^\nu\\n_1-n_2+n_3=n}}
\frac{3^{2k-3}\times2C_1^{2k-1}(|\omega|s)^{k-1}}{(k-1)!}\\
&\sum_{\substack{m^{(k-1)}=(m_j)_{1\leq j\leq2k-1}\in(\bZ^\nu)^{2k-1}\\\cc\ca\cs(m^{(k-1)})=n_1}}\sum_{\alpha=(\alpha_j)_{1\leq j\leq2k-1}\in\tG^{(k-1)}}\prod_{j=1}^{2k-1}|m_j|^{\alpha_j}
e^{-\rho|m_j|}C_1e^{-\rho|n_2|}C_1e^{-\rho|n_3|}{\rm d}s\\
\ll&\frac{3^{2k-3}\times 2C_1^{2k+1}(|\omega|t)^k}{k!}\sum_{\substack{m^{(k)}=(m_j)_{1\leq j\leq 2k+1}\in(\bZ^\nu)^{2k+1}\\\cc\ca\cs(m^{(k)})=n}}\sum_{\alpha\in\tG^{(k)}}\prod_{j=1}^{2k+1}|m_j|^{\alpha_j}
e^{-\rho|m_j|}.
\end{align*}
Analogously, for $j=2,3$, we have
\begin{align*}
\Psi_j\ll\frac{3^{2k-3}\times 2C_1^{2k+1}(|\omega|t)^k}{k!}\sum_{\substack{m^{(k)}=(m_j)_{1\leq j\leq 2k+1}\in(\bZ^\nu)^{2k+1}\\\cc\ca\cs(m^{(k)})=n}}\sum_{\alpha\in\tG^{(k)}}\prod_{j=1}^{2k+1}|m_j|^{\alpha_j}
e^{-\rho|m_j|}.
\end{align*}
Thus
\begin{align*}
\cc(t,n)-\cd(t,n)\ll\frac{3^{2k-1}\times 2C_1^{2k+1}(|\omega|t)^k}{k!}\sum_{\substack{m^{(k)}=(m_j)_{1\leq j\leq 2k+1}\in(\bZ^\nu)^{2k+1}\\\cc\ca\cs(m^{(k)})=n}}\sum_{\alpha\in\tG^{(k)}}\prod_{j=1}^{2k+1}|m_j|^{\alpha_j}
e^{-\rho|m_j|}.
\end{align*}
This shows that \eqref{cddd} is true for $k$, and hence for all $k\geq1$ by induction.

It follows from
\begin{align*}
\frac{1}{k!}\sum_{\substack{m^{(k)}=(m_j)_{1\leq j\leq 2k+1}\in(\bZ^\nu)^{2k+1}\\\cc\ca\cs(m^{(k)})=n}}\sum_{\alpha\in\tG^{(k)}}\prod_{j=1}^{2k+1}|m_j|^{\alpha_j}
e^{-\rho|m_j|}\ll e^{-\frac{\rho}{2}|n|}(12\rho^{-1})^\nu e^{1/2}\left(4e(12\rho^{-1})^{2\nu+1}\right)^k
\end{align*}
that
\[\cc(t,n)-\cd(t,n)\ll \frac{2}{3}C_1e^{1/2}(12\rho^{-1})^\nu (36C_1^2e(12\rho^{-1})^{2\nu+1}|\omega|t)^ke^{-\frac{\rho}{2}|n|}.\]
Hence
\[\cc(t,n)\equiv\cd(t,n),\quad\forall n\in\bZ^\nu,\]
provided that
\begin{align}\label{t4}
0<t<\min\left\{t_1,\frac{1}{36C_1^2e(12\rho^{-1})^{2\nu+1}|\omega|}\right\}:=t_4.
\end{align}

\end{proof}

\subsubsection{Asymptotic Dynamics}
In this subsubsection we prove that, in a weakly nonlinear setting, within the given time scale, the nonlinear solution will be asymptotic to the associated linear solution in the sense of both sup-norm and analytic Sobolev-norm.

Clearly the linear solution is given by the following Fourier series
$$u_{\text{linear}}(t,x)=\sum_{n\in\mathbb Z^\nu}e^{-{\rm i}\langle n\rangle^2t}\cc(n)e^{{\rm i}\langle n\rangle x}.$$

For the asymptotic dynamics in the sense of sup-norm, it follows from the uniform-in-time decay of the Fourier coefficients that
\begin{align*}
\|(u-u_{\text{linear}})(t)\|_{L_x^\infty(\mathbb R)}&\leq|\epsilon||\omega|\sum_{n\in\mathbb Z^\nu}|n|\int_0^t\sum_{\substack{n_j\in\bZ^\nu,~~j=1,\cdots,3\\\sum_{j=1}^{3}(-1)^{j-1}n_j=n}}
\prod_{j=1}^{3}|\cc(s,n_j)|{\rm d}s\\
&\leq|\epsilon||\omega|t\sum_{n\in\mathbb Z^\nu}\sum_{\substack{n_j\in\mathbb Z^\nu, j=1,\cdots, 3\\n_1-n_2+n_3=n}}\prod_{j=1}^{3}e^{-\frac{\kappa}{2}|n_j|}\\
&\lesssim |\epsilon|^{\eta},
\end{align*}
where $t=|\epsilon|^{-1+\eta}$ with $0<\eta\ll1$. This implies that
\[\|(u-u_{\text{linear}})(t)\|_{L_x^\infty(\mathbb R)}\rightarrow0,\quad{\text{as}}~\epsilon\rightarrow0.\]

For the asymptotic dynamics in the sense of analytic Sobolev-norm, it follows from the uniform-in-time decay of the Fourier coefficients that
\begin{align*}
\|(u-u_{\text{linear}})(t)\|^2_{\mathcal H_x^{\varrho}(\mathbb R)}
&=\|e^{\varrho|n|}(\widehat u-\widehat{u_{\text{linear}}})(t)\|^2_{\ell_{n}^2(\mathbb Z^\nu)}\\
&=\sum_{n\in\mathbb Z^\nu}e^{2\varrho|n|}|(\widehat u-\widehat{u_{\text{linear}}})(t)|^2\\
&=\sum_{n\in\mathbb Z^\nu}e^{2\varrho|n|}|\cc(t,n)-e^{-{\rm i}\langle n\rangle^2t}\cc(n)|^2\\
&\leq|\epsilon|^2|\omega|^2\sum_{n\in\mathbb Z^\nu}|n|^2e^{2\varrho|n|}\left\{\int_0^t
\sum_{\substack{n_1,n_2,n_3\in\bZ^\nu\\n_1-n_2+n_3=n}}\prod_{j=1}^{3}|\cc(s,n_j)|{\rm d}s\right\}^2\\
&\lesssim|\epsilon|^2\sum_{n\in\mathbb Z^\nu}|n|^2e^{2\varrho|n|}\left\{\int_0^t
\sum_{\substack{n_1,n_2,n_3\in\bZ^\nu\\n_1-n_2+n_3=n}}\prod_{j=1}^{3}e^{-\frac{\kappa}{2}|n_j|}{\rm d}s\right\}^2\\
&\leq(|\epsilon|t)^2\sum_{n\in\mathbb Z^\nu}|n|^2e^{2\varrho|n|}\left\{
\sum_{\substack{n_1,n_2,n_3\in\bZ^\nu\\n_1-n_2+n_3=n}}\prod_{j=1}^{3}e^{-\frac{\kappa}{4}|n_j|}\cdot \prod_{j=1}^{3}e^{-\frac{\kappa}{4}|n_j|}\right\}^2\\
&=(|\epsilon|t)^2\sum_{n\in\mathbb Z^\nu}|n|^2e^{2\varrho|n|}\left\{
\sum_{\substack{n_1,n_2,n_3\in\bZ^\nu\\n_1-n_2+n_3=n}}\prod_{j=1}^{3}e^{-\frac{\kappa}{4}|n_j|}\cdot e^{-\frac{\kappa}{4}\sum_{j=1}^3|n_j|}\right\}^2\\
&\leq(|\epsilon|t)^2\sum_{n\in\mathbb Z^\nu}|n|^2e^{2\varrho|n|}\left\{
\sum_{\substack{n_1,n_2,n_3\in\bZ^\nu\\n_1-n_2+n_3=n}}\prod_{j=1}^{3}e^{-\frac{\kappa}{4}|n_j|}\cdot e^{-\frac{\kappa}{4}|n_1-n_2+n_3|}\right\}^2\\
&=(|\epsilon|t)^2\sum_{n\in\mathbb Z^\nu}|n|^2e^{2\varrho|n|}\left\{
\sum_{\substack{n_1,n_2,n_3\in\bZ^\nu\\n_1-n_2+n_3=n}}\prod_{j=1}^{3}e^{-\frac{\kappa}{4}|n_j|}\cdot e^{-\frac{\kappa}{4}|n|}\right\}^2\\
&\leq(|\epsilon|t)^2\sum_{n\in\mathbb Z^\nu}|n|^2e^{-\left(\frac{\kappa}{2}-2\varrho\right)|n|}\left\{
\sum_{n_1,n_2,n_3\in\mathbb Z^\nu}\prod_{j=1}^{3}e^{-\frac{\kappa}{4}|n_j|}\right\}^2\\
&=(|\epsilon|t)^2\sum_{n\in\mathbb Z^\nu}|n|^2e^{-\left(\frac{\kappa}{2}-2\varrho\right)|n|}
\left\{
\prod_{j=1}^3\sum_{n_j\in\mathbb Z^\nu}e^{-\frac{\kappa}{4}|n_j|}\right\}^2\\
&\leq(12\kappa^{-1})^6(|\epsilon|t)^2\sum_{n\in\mathbb Z^\nu}\underbrace{|n|^2e^{-\left(\frac{\kappa}{4}-4\varrho\right)|n|}}_{\text{uniform bounded}}\cdot e^{-\left(\frac{\kappa}{4}-4\varrho\right)|n|}\\
&\leq3(\kappa/4-4\rho)^{-1}(12\kappa^{-1})^6(|\epsilon|t)^2\\
&\lesssim|\epsilon|^{2\eta},
\end{align*}
provided that $t=|\epsilon|^{-1+\eta}$ with $0<\eta\ll 1$, and
\[0<\frac{\kappa}{4}-4\varrho\leq1.\]
This shows that
\[\|(u-u_{\text{linear}})(t)\|_{H_x^{\varrho}(\mathbb R)}\rightarrow0,\quad\text{as}~\epsilon\rightarrow 0.\]
The proof of Theorem \ref{dnlsth} is competed.

\subsection{gdNLS}

Consider the following generalized derivative nonlinear Schr\"odinger (gdNLS for short) equation
\begin{align}\label{gdnls}
\tag{gdNLS}{\rm i}\partial_tu+\partial_{xx}u+{\rm i}\partial_x(|u|^{2p}u)=0
\end{align}
with quasi-periodic initial data \eqref{id},
where $p\in\bN$ is a parameter describing the strength of the nonlinearity; see \cite{CG2017MMM}.

The case $p=1$ in \eqref{gdnls} is the classical derivative NLS \eqref{dnls}, as discussed above.

For the general case, if there exists $(\calA,\fr)\in(0,\infty)\times(0,1]$ such that
\begin{align}
c(n)\ll\calA^{\frac{1}{2p}} e^{-\fr|n|},\quad\forall n\in\bZ^\nu,
\end{align}
using the combinatorial analysis method with Feynman diagram and {\bf pcc} label method,
we can prove that
quasi-periodic Cauchy problem \eqref{gdnls}-\eqref{id} has a unique spatially quasi-periodic solution, locally in time and globally in space, with the same frequency vector (it retains the same spatial quasi-periodicity), in the (classical, analytic) sense w.r.t. (time, space), and it is asymptotic to the linear solution within the time scale in a weakly nonlinear setting. 

\section{Appendix}
\begin{lemm}\label{l2l}
For any given $1\leq m\in\mathbb N$ and $K>0$, we have
\begin{align}\label{eq}
y^me^{-Ky}\leq m!(K^{-1})^m,\quad \forall y\geq0.
\end{align}
\end{lemm}

%


\begin{lemm}\cite[Lemma 9.1]{MR4751185}
\label{l1l}
If $0<K\leq1$, then
$$\sum_{m\in\mathbb Z}e^{-K|m|}\leq 3K^{-1}.$$
\end{lemm}

%

\begin{lemm}(\cite[Lemma 2.9 (2)]{DG2017JAMS})\label{3}
For $0 < \kappa \leq 1$, we have
$$
\sum_{(m_1,\cdots,m_r)\in(\mathbb Z^{\nu})^r}\prod_{j=1}^{r}|m_j|^{\alpha_j}e^{-\kappa|m_j|}\leq(6\kappa^{-1})^{|\alpha|+\nu r}\prod_{j=1}^{r}\alpha_j!.
$$
\end{lemm}


\begin{lemm}(The length of an element in $\cR^{(k,\gamma^{(k)})}$)\label{lemalpha}
For any multi-index $\alpha\in\cR^{(k,\gamma^{(k)})}$, we have
\begin{align}\label{alpha}
|\alpha|=\ell(\gamma^{(k)}).
\end{align}
\end{lemm}
\begin{proof}
It is clear that \eqref{alpha} holds true for $0=\gamma^{(k)}\in\Gamma^{(k)} (k\geq1)$ and $1=\gamma^{(1)}\in\Gamma^{(1)}$ by the definitions of $\cR$ and $\ell$. This shows that \eqref{alpha} is holds for $k=1$.

Let $k\geq2$. Assume that \eqref{alpha} is true for all $1<k^\prime<k$. For $k$, we need to consider only $\gamma^{(k)}=(\gamma_j^{(k-1)})_{1\leq j\leq3}\in\prod_{j=1}^3\Gamma^{(k-1)}$. For any $\alpha\in\cR^{(k,\gamma^{(k)})}$, there exist $\beta\in\ca^{(k,\gamma^{(k)})}$ and $\alpha_j\in\cR^{(k-1,\gamma_{j}^{(k-1)})} (j=1,2,3)$ such that $\alpha=\beta+(\alpha_1,\alpha_2,\alpha_3)$. Hence $|\alpha|=|\beta|+\sum_{j=1}^2|\alpha_{j}|=1+\sum_{j=1}^3\ell(\gamma_j^{(k-1)})=\ell(\gamma^{(k)})$. This proves that \eqref{alpha} holds for $k$ and all $k\geq1$ by induction, finishing the proof of Lemma \ref{lemalpha}.
\end{proof}

\begin{lemm}(The length of an element in $\tG^{(k)}$)\label{lemg}
For any multi-index $\alpha\in\tG^{(k)}$, where $k\geq1$, we have
\begin{align}\label{alphaggg}
\alpha\in\bR^{2k+1}\quad\text{and}\quad |\alpha|=k.
\end{align}
\end{lemm}
\begin{proof}
Recalling the definition of $\tG$, it is easy to see that \eqref{alphaggg} is true for $k=1$.

Let $k\geq2$. Assume \eqref{alphaggg} holds for all $1<k^\prime<k$. For $k$ and any $\alpha\in\tG^{(k)}$, there exist $\beta\in\ce^{(k)}$ and $\alpha^\prime\in\tG^{(k-1)}$ such that $\alpha=\beta+(\alpha^\prime,0,0)$. Clearly, $\beta\in\bR^{2k+1}$ and $|\beta|=1$. By induction hypothesis, $\alpha^\prime\in\bR^{2k-1}$ and $|\alpha^\prime|=k-1$. Hence $\alpha\in\bR^{2k+1}$ and $|\alpha|=|\beta|+|\alpha^\prime|=k$. This shows that \eqref{alphaggg} holds for $k$, and for all $k\geq1$ by induction. This finishes the proof of Lemma \ref{lemg}.
\end{proof}

\begin{lemm}(\cite[Lemma 2.18]{DG2017JAMS})
\label{aaaa}
Let $N,L$ be arbitrary. Set
\[\cA_N(L):=\{\alpha=(\alpha_j)_{1\leq j\leq N}\in\bN^N: |\alpha|=L\}.\]
Then
\begin{align*}
\sum_{\alpha=(\alpha_j)_{1\leq j\leq N}\in\cA_N(L)}\prod_{j=1}^{N}\alpha_j!<(2N)^{L}.
\end{align*}
\end{lemm}
\begin{lemm}(Stirling's formula)\label{stf}
For all $n\in\bN$,
\[n!=\sqrt{2\pi n}\left(\frac{n}{e}\right)^n\cdot e^{\frac{\theta}{12n}},\quad 0<\theta<1.\]
Hence
\begin{align}\label{prod}
n!\geq\sqrt{2\pi n}\left(\frac{n}{e}\right)^n\geq\left(\frac{n}{e}\right)^n.
\end{align}
\end{lemm}


\bibliographystyle{alpha}
\bibliography{dNLS}

\end{document}